\documentclass[leqno]{amsart}

\usepackage{amsmath}
\usepackage{amsfonts}
\usepackage{amssymb}
\usepackage{pdfsync}
\usepackage{color,graphicx,tikz-cd}
\usepackage{a4wide}
\usepackage{amscd,amsthm,latexsym}
\usepackage{hyperref}
\usepackage{relsize}
\usepackage[misc]{ifsym}
\usepackage{cite}

\title{Reverse plane partitions of skew staircase shapes and $q$-Euler numbers}

\author{Byung-Hak Hwang}
\address{Department of Mathematics, Seoul National University, Seoul,
South Korea}
\email{xoda@snu.ac.kr}

\author{Jang Soo Kim}
\address{
Department of Mathematics, Sungkyunkwan University, Suwon,
South Korea}
\email{jangsookim@skku.edu}

\author{Meesue Yoo}
\address{
Applied Algebra and Optimization Research Center, Sungkyunkwan University, Suwon,
South Korea}
\email{meesue.yoo@skku.edu (\Letter)}

\author{Sun-mi Yun}
\address{
Department of Mathematics, Sungkyunkwan University, Suwon,
South Korea}
\email{sera314@skku.edu}

\date{\today}

\thanks{The second author was supported by NRF grants \#2016R1D1A1A09917506 and \#2016R1A5A1008055.
The third author was supported by NRF grants \#2016R1A5A1008055 and \#2017R1C1B2005653.}



\date{\today}

\newtheorem{thm}{Theorem}[section]
\newtheorem{lem}[thm]{Lemma}
\newtheorem{prop}[thm]{Proposition}
\newtheorem{cor}[thm]{Corollary}
\theoremstyle{definition}
\newtheorem{exam}[thm]{Example}

\newtheorem{remark}[thm]{Remark}

\newcommand\PP{\mathcal{P}}

\newcommand\ZZ{\mathbb{Z}}

\newcommand\EE{\mathcal{E}}

\newcommand\VV{\mathcal{V}}
\newcommand\HP{\mathcal{HP}}

\newcommand\RPP{{\operatorname{RPP}}}
\newcommand\SSYT{{\operatorname{SSYT}}}
\newcommand\TAB{{\operatorname{TAB}}}
\newcommand\ST{{\operatorname{ST}}}

\newcommand\Sym{\mathfrak{S}}
\newcommand\lm{{\lambda/\mu}}

\newcommand\inv{\operatorname{inv}}
\newcommand\maj{\operatorname{maj}}
\newcommand\des{\operatorname{des}}
\newcommand\ndes{\operatorname{ndes}}
\newcommand\asc{\operatorname{asc}}
\newcommand\nasc{\operatorname{nasc}}
\newcommand\Des{\operatorname{Des}}
\newcommand\NDes{\operatorname{NDes}}
\newcommand\Asc{\operatorname{Asc}}
\newcommand\NAsc{\operatorname{NAsc}}
\newcommand\ND{\mathcal{ND}}

\newcommand\HT{\operatorname{ht}}
\newcommand\wt{\operatorname{wt}}
\newcommand\wtext{\operatorname{wt_{ext}}}

\newcommand\sgn{\operatorname{sgn}}
\newcommand\Dyck{\operatorname{Dyck}}
\newcommand\Sch{{\operatorname{Sch}}}

\newcommand\Alt{\operatorname{Alt}}
\newcommand\Ralt{\operatorname{Ralt}}
\newcommand\FA{F_{\operatorname{alt}}}
\newcommand\hs{\operatorname{hs}}

\newcommand\flr[1]{\left\lfloor #1\right\rfloor}
\newcommand\Qbinom[3]{\genfrac{[}{]}{0pt}{}{#1}{#2}_{#3}}
\newcommand\qbinom[2]{\Qbinom{#1}{#2}{q}}
\newcommand\LL{\mathcal{L}}
\newcommand\NN{\mathbb{N}}
\newcommand\qand{\qquad \mbox{and} \qquad}

\begin{document}

\begin{abstract}
Recently, Naruse discovered a hook length formula for the number of standard Young tableaux of a skew shape. 
Morales, Pak and Panova found two $q$-analogs of Naruse's hook length formula over semistandard Young tableaux (SSYTs) and reverse plane partitions (RPPs). 
As an application of their formula, they expressed certain $q$-Euler numbers,
which are generating functions for SSYTs and RPPs of a zigzag border strip, in terms of weighted Dyck paths. They found a determinantal formula for the generating function for SSYTs of a skew staircase shape and proposed two conjectures related to RPPs of the same shape. One conjecture is a determinantal formula for  the number of \emph{pleasant diagrams}  in terms of Schr\"oder paths and the other conjecture is a determinantal formula for the generating function for RPPs of a skew staircase shape in terms of $q$-Euler numbers. 

In this paper, we show that the results of Morales, Pak and Panova on the $q$-Euler numbers can be derived from previously known results due to Prodinger by manipulating continued fractions. These $q$-Euler numbers are naturally expressed as generating functions for alternating permutations with certain statistics involving \emph{maj}. It has been proved by Huber and Yee that these $q$-Euler numbers are generating functions for alternating permutations with certain statistics involving \emph{inv}. By modifying Foata's bijection we construct a bijection on alternating permutations which sends the statistics involving \emph{maj} to the statistic involving \emph{inv}. We also prove the aforementioned two conjectures of Morales, Pak and Panova. 
\end{abstract}

\maketitle
\tableofcontents

\section{Introduction}

The classical hook length formula due to Frame et al.~\cite{Frame1954} states that,
for a partition $\lambda$ of $n$, the number $f^{\lambda}$ of standard Young tableaux of shape $\lambda$ is given by
\[
f^{\lambda} = n! \prod_{u\in \lambda}\frac{1}{h(u)},
\]
where $h(u)$ is the \emph{hook length} of the square $u$.
See Section~\ref{sec:preliminaries} for all undefined terminologies in the introduction.

Recently, Naruse \cite{Naruse} generalized the hook length formula to standard Young tableaux of a skew shape as follows.
For partitions $\mu\subset \lambda$, the number $f^{\lm}$ of standard Young tableaux of shape $\lm$ is given by
\begin{equation}
  \label{eq:naruse}
f^{\lm} = |\lm|! \sum_{D\in\EE(\lm)} \prod_{u\in \lambda\setminus D}\frac{1}{h(u)},  
\end{equation}
where $\EE(\lm)$ is the set of \emph{excited diagrams} of $\lm$. 

Morales et al.~\cite{MPP1} found two natural $q$-analogs of Naruse's hook length formula over \emph{semistandard Young tableaux} (SSYTs) and \emph{reverse plane partitions} (RPPs):
\begin{equation}
  \label{eq:MPP1}
\sum_{\pi\in\SSYT(\lm)}q^{|\pi|} =\sum_{D\in\EE(\lm)} \prod_{(i,j)\in\lambda\setminus D}\frac{q^{\lambda_j'-i}}{1-q^{h(i,j)}}
\end{equation}
and
\begin{equation}
  \label{eq:MPP2}
\sum_{\pi\in\RPP(\lm)}q^{|\pi|} = \sum_{P\in\PP(\lm)} \prod_{u\in P} \frac{q^{h(u)}}{1-q^{h(u)}},
\end{equation}
where $\SSYT(\lm)$ is the set of SSYTs of shape $\lm$, $\RPP(\lm)$ is the set of RPPs of shape $\lm$ and $\PP(\lm)$ is the set of \emph{pleasant diagrams} of $\lm$. 

Let $\delta_n$ denote the staircase shape partition $(n-1,n-2,\dots,1)$. Morales et al. \cite{MPP2} found interesting connections between SSYTs or RPPs of shape $\lm$ with $q$-Euler numbers when $\lm$ is the skew staircase shape $\delta_{n+2k}/\delta_n$. To explain their results, we need some definitions.

Let $\Alt_{n}$ be the set of alternating permutations of $\{1,2,\dots,n\}$, where $\pi=\pi_1\pi_2\dots\pi_n$ is called \emph{alternating} if 
$\pi_1<\pi_2>\pi_3<\pi_4>\cdots$.  The \emph{Euler number} $E_n$ is defined to be the cardinality of $\Alt_{n}$. It is well known that
\[
\sum_{n\ge0} E_n\frac{x^n}{n!} = \sec x + \tan x.
\]
The even-indexed Euler numbers $E_{2n}$ are called the \emph{secant numbers} and the odd-indexed Euler numbers $E_{2n+1}$ are called the \emph{tangent numbers}.
We refer the reader to \cite{Stanley2010} for many interesting properties of alternating permutations and Euler numbers. 

When $\lambda/\mu=\delta_{n+2}/\delta_n$, the right hand sides of \eqref{eq:MPP1} and \eqref{eq:MPP2} are closely related to the $q$-Euler numbers $E_{n}(q)$ and $E^*_{n}(q)$ defined by
\begin{equation}
  \label{eq:En_maj}
E_{n}(q)=\sum_{\pi\in\Alt_{n}} q^{\maj(\pi^{-1})} \qand E^*_{n}(q)=\sum_{\pi\in\Alt_{n}} q^{\maj(\kappa_{n}\pi^{-1})},  
\end{equation}
where $\kappa_{n}$ is the permutation 
\[
(1)(2,3)(4,5)\dots(2\flr{(n-1)/2}, 2\flr{(n-1)/2}+1)
\]
 in cycle notation and $\maj(\pi)$ is the \emph{major index} of $\pi$.  As corollaries of  \eqref{eq:MPP1} and \eqref{eq:MPP2} for $\lambda/\mu=\delta_{n+2}/\delta_n$, 
Morales et al. \cite[Corollaries~1.7 and 1.8]{MPP2} obtained that
\begin{equation}
  \label{eq:MPP_Euler}
\frac{E_{2n+1}(q)}{(q;q)_{2n+1}} 
= \sum_{D\in\Dyck_{2n}}  \prod_{(a,b)\in D} \frac{q^b}{1-q^{2b+1}}
\end{equation}
and
\begin{equation}
  \label{eq:MPP_Euler*}
\frac{E^*_{2n+1}(q)}{(q;q)_{2n+1}} 
= \sum_{D\in\Dyck_{2n}} q^{H(D)} \prod_{(a,b)\in D} \frac{1}{1-q^{2b+1}},
\end{equation}
where $\Dyck_{2n}$ is the set of \emph{Dyck paths} of length $2n$, $H(D)=\sum_{(a,b)\in\HP(D)}(2b+1)$, $\HP(D)$ is the set of \emph{high peaks} in $D$ and 
\[
(a;q)_n = (1-a)(1-aq)\cdots(1-aq^{n-1}).
\]

The $q$-Euler numbers $E_n(q)$ and $E_n^*(q)$ have been studied by Prodinger \cite{Prodinger2008} in different contexts.
In this paper we give different proofs of \eqref{eq:MPP_Euler}  and \eqref{eq:MPP_Euler*} using Prodinger's results and continued fractions.

In fact, the $q$-Euler numbers $E_n(q)$ were first introduced by Jackson \cite{Jackson1904} and they have the following combinatorial interpretation, 
see \cite{Stanley76, Stanley2010}:
\begin{equation}
  \label{eq:En_inv}
E_{n}(q)=\sum_{\pi\in\Alt_{n}} q^{\inv(\pi)},
\end{equation}
where $\inv(\pi)$ is the number of \emph{inversions} of $\pi$. 
Observe that, by \eqref{eq:En_maj} and \eqref{eq:En_inv}, we have
\begin{equation}
  \label{eq:inv_maj}
\sum_{\pi\in\Alt_{n}} q^{\maj(\pi^{-1})} = \sum_{\pi\in\Alt_{n}} q^{\inv(\pi)}.  
\end{equation}
Foata's bijection \cite{Foata68} gives a bijective proof of \eqref{eq:inv_maj}.

The $q$-Euler numbers $E^*_n(q)$ have also been studied by Huber and Yee \cite{Huber2010} who showed that
\begin{equation}
  \label{eq:Huber}
E_{2n+1}^*(q)=\sum_{\pi\in\Alt_{2n+1}} q^{\inv(\pi)-\ndes(\pi_e)},
\end{equation}
where $\ndes(\pi_e)$ is the number of \emph{non-descents} in $\pi_e=\pi_2\pi_4\cdots\pi_{\flr{n/2}}$.
Again, by \eqref{eq:En_maj} and \eqref{eq:Huber}, we have
\begin{equation}
  \label{eq:inv_maj*}
\sum_{\pi\in\Alt_{2n+1}} q^{\maj(\kappa_{2n+1}\pi^{-1})} = \sum_{\pi\in\Alt_{2n+1}} q^{\inv(\pi)-\ndes(\pi_e)}.  
\end{equation}
In this paper we give a bijective proof of \eqref{eq:inv_maj*}  by modifying Foata's bijection. 

There are similar $q$-Euler numbers studied in \cite{Prodinger2008, Huber2010}.
We show that these $q$-Euler numbers have similar properties, see Theorems~\ref{thm:q-tan} and \ref{thm:q-sec}.

Note that in \eqref{eq:MPP1} the generating function for SSYTs of shape $\lm$ is expressed in terms of the excited diagrams of $\lm$. Morales et al. \cite[Corollaries~8.1~and~8.8]{MPP2} proved the following determinantal formulas for the number $e(\lm)$ of excited diagrams of $\lm$ and the generating function for SSYTs of shape $\lm$ for $\lm=\delta_{n+2k}/\delta_n$:
\begin{equation}
  \label{eq:excited}
e(\delta_{n+2k}/\delta_n) = \det(C_{n+i+j-2})_{i,j=1}^k 
= \prod_{1\le i<j\le n} \frac{2k+i+j-1}{i+j-1},
\end{equation}
where $C_n=\frac{1}{n+1}\binom{2n}{n}$ is the Catalan number, and 
\begin{equation}
  \label{eq:ssyt}
\sum_{\pi\in\SSYT(\delta_{n+2k}/\delta_n)} q^{|\pi|}
=\det\left( \frac{E_{2n+2i+2j-3}(q)}{(q;q)_{2n+2i+2j-3}} \right)_{i,j=1}^k.
\end{equation}

Let $p(\lm)$ be the number of pleasant diagrams of $\lm$. 
Morales et al. \cite{MPP2} showed that $p(\delta_{n+2}/\delta_n)=\mathfrak{s}_n$, where
$\mathfrak{s}_n=2^{n+2} s_n$ for the \emph{little Schr\"oder number} $s_n$. They proposed the following conjectures on  $p(\lm)$ and the generating function for RPPs of shape $\lm$ for $\lm=\delta_{n+2k}/\delta_n$.

\begin{thm} \cite[Conjecture 9.3]{MPP2} \label{conj:9.3}
We have
\[
p(\delta_{n+2k}/\delta_n) = 2^{\binom k2} \det(\mathfrak{s}_{n-2+i+j})_{i,j=1}^k.
\]  
\end{thm}

\begin{thm}\cite[Conjecture 9.6]{MPP2} \label{conj:9.6}
We have
\[
\sum_{\pi\in \RPP(\delta_{n+2k}/\delta_n)} q^{|\pi|} 
=q^{-\frac{k(k-1)(6n+8k-1)}{6}} \det\left( \frac{E^*_{2n+2i+2j-3}(q)}{(q;q)_{2n+2i+2j-3}} \right)_{i,j=1}^k.
\]  
\end{thm}

In this paper we prove their conjectures. We remark that the determinants in \eqref{eq:excited} and \eqref{eq:ssyt} can be expressed in terms of non-intersecting paths using Lindstr\"om--Gessel--Viennot lemma \cite{Lindstrom, GesselViennot}. However, the determinants in Theorems~\ref{conj:9.3} and \ref{conj:9.6} are not of such forms that we can directly apply Lindstr\"om--Gessel--Viennot lemma. 
Therefore, we need extra work to relate them with  non-intersecting paths. To this end we first express the determinants in terms of weakly non-intersecting paths. Then we resolve each pair of weakly non-intersecting paths into a pair of strictly non-intersecting paths one at a time. 

The remainder of this paper is organized as follows. In Section~\ref{sec:preliminaries} we provide necessary definitions and some known results. In Section~\ref{sec:q-euler} we give different proofs of \eqref{eq:MPP_Euler} and \eqref{eq:MPP_Euler*} using Prodinger's results, lattice paths and continued fractions. In Section~\ref{sec:foata} we state several properties of Prodinger's $q$-Euler numbers and give a bijective proof of \eqref{eq:inv_maj*} by constructing a Foata-type bijection. In Section~\ref{sec:MPP conjectures} we prove Theorems~\ref{conj:9.3} and \ref{conj:9.6}. 
In Section~\ref{sec:lascoux-pragacz-type}, we find a determinantal formula for $p(\lm)$ and the generating function for the reverse plane partitions of shape $\lm$ for a certain class of skew shapes $\lm$ including $\delta_{n+2k}/\delta_{n}$ and $\delta_{n+2k+1}/\delta_{n}$. 

\section{Preliminaries}
\label{sec:preliminaries}

In this section we give basic definitions and some known results.

\subsection{Permutation statistics and alternating permutations}

The set of integers is denoted by $\ZZ$ and the set of nonnegative integers is denoted by $\NN$. 
Let $[n]:=\{1,2,\dots,n\}$. 

A \emph{permutation} of $[n]$ is a bijection $\pi:[n]\to[n]$.  The set of permutations of $[n]$ is denoted by $\Sym_n$.  
We also consider a permutation $\pi\in \Sym_n$ as the word $\pi=\pi_1\pi_2\dots\pi_n$ of integers, where $\pi_i=\pi(i)$ for $i\in[n]$.
For $\pi,\sigma\in \Sym_n$, the product $\pi\sigma$ is defined as the usual composition of functions, i.e., $\pi\sigma(i) = \pi(\sigma(i))$ for $i\in[n]$. 

For a permutation $\pi=\pi_1\pi_2\dots\pi_n\in \Sym_n$ we define
\[
\pi_o = \pi_1\pi_3\cdots \pi_{2\flr{(n-1)/2}+1} \qand
\pi_e = \pi_2\pi_4\cdots \pi_{2\flr{n/2}}.
\]

Let $w=w_1w_2\dots w_n$ be a word of length $n$ consisting of integers. A \emph{descent} (resp.~\emph{ascent}) of $w$ is an integer $i\in[n-1]$ satisfying $w_i>w_{i+1}$ (resp.~$w_i<w_{i+1}$). A \emph{non-descent} (resp.~\emph{non-ascent}) of $w$ is an integer $i\in[n]$ that is not a descent (resp.~\emph{ascent}). In other words, $i$ is a non-descent (resp.~non-ascent) of $w$ if and only if either $i$ is an ascent (resp.~descent) of $w$ or $i=n$. We denote by $\Des(w)$, $\Asc(w)$, $\NDes(w)$ and $\NAsc(w)$ the sets of descents, ascents, non-descents and non-ascents of $w$ respectively. The \emph{major index} $\maj(w)$ of $w$ is defined to be the sum of descents of $w$.  An \emph{inversion} of $w$ is a pair $(i,j)$ of integers $1\le i<j\le n$ satisfying $w_i>w_j$. The number of inversions of $w$ is denoted by $\inv(w)$.

An \emph{alternating permutation} (resp.~\emph{reverse alternating permutation}) is a permutation $\pi=\pi_1\pi_2\dots\pi_n\in \Sym_n$
satisfying $\pi_1<\pi_2>\pi_3<\pi_4>\cdots$ (resp.~$\pi_1>\pi_2<\pi_3>\pi_4<\cdots$). 
The set of alternating permutations (resp.~reverse alternating permutations) in $\Sym_n$ is denoted by $\Alt_{n}$ (resp.~$\Ralt_{n}$).
Note that, for $\pi\in \Sym_n$, we have $\pi\in \Alt_{n}$ (resp.~$\pi\in \Ralt_{n}$) if and only if $\Des(\pi)=\{2,4,6,\dots\}\cap [n-1]$
(resp.~$\Des(\pi)=\{1,3,5,\dots\}\cap [n-1]$).

For $n\ge1$, we define
\begin{align*}
\eta_{2n} &= (1,2)(3,4)\cdots (2n-1,2n)\in \Sym_{2n},\\
\eta_{2n+1} &= (1,2)(3,4)\cdots (2n-1,2n)(2n+1) \in \Sym_{2n+1},\\
\kappa_{2n} &= (1)(2,3)(4,5)\cdots (2n-2,2n-1)(2n)\in \Sym_{2n},\\
\kappa_{2n+1} &= (1)(2,3)(4,5)\cdots (2n-1,2n) \in \Sym_{2n+1}.
\end{align*}

\subsection{$(P,\omega)$-partitions}

$(P,\omega)$-partitions are generalizations of partitions introduced by Stanley\cite{Stanley72}. 
Here we recall basic properties of $(P,\omega)$-partitions. 
We refer the reader to \cite[Chapter 3]{EC1} for more details on the theory of $(P,\omega)$-partitions. 

Let $P$ be a poset with $n$ elements. A \emph{labeling} of $P$ is a
bijection $\omega:P\to[n]$. A pair $(P,\omega)$ of a poset $P$ and its
labeling $\omega$ is called a \emph{labeled poset}.  
A \emph{$(P,\omega)$-partition} is a
function $\sigma:P\to \{0,1,2,\dots\}$ satisfying
\begin{itemize}
\item $\sigma(x)\ge \sigma(y)$ if $x\le_P y$,
\item $\sigma(x)>\sigma(y)$ if $x\le_P y$ and $\omega(x)> \omega(y)$.
\end{itemize}
We denote by $\PP(P,\omega)$ the set of $(P,\omega)$-partitions. The \emph{size} $|\sigma|$ of $\sigma\in\PP(P,\omega)$ is defined by 
\[
|\sigma| = \sum_{x\in P}\sigma(x).
\]

A \emph{linear extension} of $P$ is an arrangement
$(t_1,t_2,\dots,t_n)$ of the elements in $P$ such that if $t_i<_P t_j$
then $i<j$. The \emph{Jordan-H\"older set} $\LL(P,\omega)$ of $P$ with labeling $\omega$ is
the set of permutations of the form
$\omega(t_1)\omega(t_2)\cdots \omega(t_n)$ for some linear extension
$(t_1,t_2,\dots,t_n)$ of $P$. 

It is well known \cite[Theorem~3.15.7]{EC1} that the $(P,\omega)$-partition generating function can be written in terms of linear extensions:
\begin{equation}
  \label{eq:lin_ext}
\sum_{\sigma\in\PP(P,\omega)} q^{|\sigma|}
=\frac{\sum_{\pi\in\LL(P,\omega)}q^{\maj(\pi)}}{(q;q)_n}.
\end{equation}

\subsection{Semistandard Young tableaux and reverse plane partitions}

A \emph{partition} of $n$ is a weakly decreasing sequence $\lambda=(\lambda_1,\dots,\lambda_k)$ of positive integers summing to $n$. 
If $\lambda$ is a partition of $n$, we write $\lambda\vdash n$.

We denote the staircase partition by $\delta_n = (n-1,n-2,\dots,1)$. For $a,b\in\{0,1\}$, we define
\[
\delta_n^{(a,b)} =(n-1-a,n-2,n-3,\dots,2, 1-b).
\]
In other words,
\begin{align*}
\delta_n^{(0,0)} & =\delta_n = (n-1,n-2,n-3,\dots,1),\\
\delta_n^{(1,0)} & =(n-2,n-2,n-3,\dots,1),\\
\delta_n^{(0,1)} & =(n-1,n-2,\dots,2),\\  
\delta_n^{(1,1)} & =(n-2,n-2,n-3,\dots,2).
\end{align*}

Let $\lambda=(\lambda_1,\dots,\lambda_n)$ be a partition. The \emph{Young diagram} of $\lambda$ is the left-justified array of squares in which there are $\lambda_i$ squares in row $i$, see Figure~\ref{fig:young}. We identity a partition $\lambda$ with its Young diagram.  Considering $\lambda$ as the set of squares in the Young diagram, the notation $(i,j)\in \lambda$ means that the Young diagram of $\lambda$ has a square in the $i$th row and $j$th column. 
The \emph{transpose} $\lambda'$ of $\lambda$ is the partition $\lambda'=(\lambda'_1,\dots,\lambda_k')$ where $k=\lambda_1$ and 
$\lambda_i'$ is the number of squares in the $i$th column of the Young diagram of $\lambda$. For $(i,j)\in\lambda$, the \emph{hook length} of $(i,j)$ is defined by
\[
h(i,j)=\lambda_i+\lambda_j'-i-j+1.
\]

For two partitions $\lambda$ and $\mu$, we write $\mu\subset\lambda$ if the Young diagram of $\mu$ is contained in that of $\lambda$. 
In this case we define the \emph{skew shape} $\lm$ to be the set-theoretic difference $\lambda-\mu$ of the Young diagrams, see Figure~\ref{fig:young}.
The \emph{size} $|\lm|$ of a skew shape $\lm$ is the number of squares in $\lm$.  A partition $\lambda$ is also considered as a skew shape by $\lambda=\lambda/\emptyset$.

\begin{figure}[h]
  \centering
\includegraphics{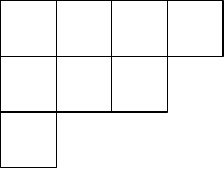}  \qquad \qquad
\includegraphics{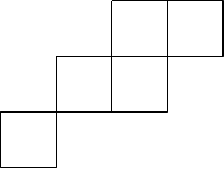}  
  \caption{The Young diagram of $(4,3,1)$ on the left and the skew shape $(4,3,1)/(2,1)$ on the right.}
  \label{fig:young}
\end{figure}

A \emph{semistandard Young tableau} (or SSYT for short) of shape $\lm$ is a filling of $\lm$ with nonnegative integers such that the integers are weakly increasing in each row and strictly increasing in each column.  A \emph{reverse plane partition} (or RPP for short) of shape $\lm$ is a filling of $\lm$ with nonnegative integers such that the integers are weakly increasing in each row and each column. A \emph{strict tableau} (or ST for short) of shape $\lm$ is a filling of $\lm$ with nonnegative integers such that the integers are strictly increasing in each row and each column. See Figure~\ref{fig:ssyt}.

\begin{figure}[h]
  \centering
\includegraphics{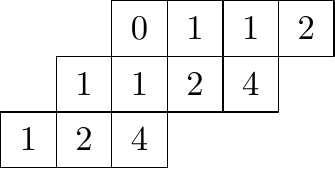}  \qquad 
\includegraphics{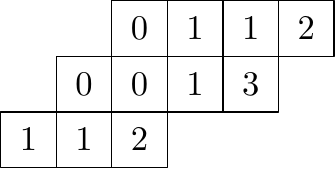} \qquad
\includegraphics{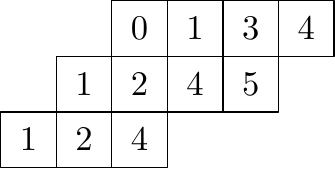}  
  \caption{A semistandard Young tableau of shape $(6,5,3)/(2,1)$ on the left, a reverse plane partition of shape $(6,5,3)/(2,1)$ in the middle and a strict tableau of shape $(6,5,3)/(2,1)$ on the right.}
  \label{fig:ssyt}
\end{figure}

 We denote by $\SSYT(\lm)$, $\RPP(\lm)$ and $\ST(\lm)$ the set of SSYTs, RPPs and STs of shape $\lm$, respectively. For an SSYT, RPP or ST $\pi$, 
the \emph{size} $|\pi|$ of $\pi$ is defined to be the sum of entries in $\pi$. 

SSYTs, RPPs and STs can be considered as $(P,\omega)$-partitions as follows. Let $\lm$ be a skew shape. 
Denote by $P_\lm$ the poset whose elements are the squares in $\lm$ with relation $x\le y$
if $x$ is weakly southeast of $y$. We denote by $\omega_\lm^\SSYT$, $\omega_\lm^\RPP$  and $\omega_\lm^\ST$ the labelings of $P_\lm$ which are uniquely determined by the following properties, see Figure~\ref{fig:labelings}. 
\begin{itemize}
\item For $(i,j),(i',j')\in\lm$, we have $\omega_\lm^\SSYT((i,j))\le \omega_\lm^\SSYT((i',j'))$ if and only if $j>j'$, or $j=j'$ and $i\le i'$.
\item For $(i,j),(i',j')\in\lm$, we have $\omega_\lm^\RPP((i,j))\le \omega_\lm^\RPP((i',j'))$ if and only if $j>j'$, or $j=j'$ and $i\ge i'$.
\item For $(i,j),(i',j')\in\lm$, we have $\omega_\lm^\ST((i,j))\le \omega_\lm^\ST((i',j'))$ if and only if $i<i'$, or $i=i'$ and $j\le j'$.
\end{itemize}

\begin{figure}[h]
  \centering
\includegraphics{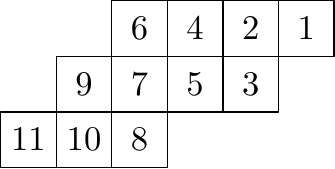}  \qquad 
\includegraphics{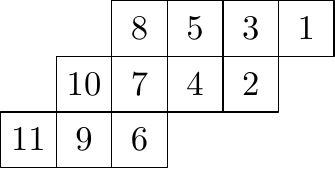}  \qquad
\includegraphics{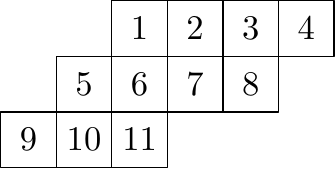}  
  \caption{The labelings $\omega_{\lm}^\SSYT$ on the left,
$\omega_{\lm}^\RPP$ in the middle and $\omega_{\lm}^\ST$ on the right for $\lm=(6,5,3)/(2,1)$.}
  \label{fig:labelings}
\end{figure}

We can naturally identify the elements in $\PP(P_\lm,\omega_\lm^\SSYT)$ (resp.~$\PP(P_\lm,\omega_\lm^\RPP)$ and $\PP(P_\lm,\omega_\lm^\ST)$)
with the SSYTs (resp.~RPPs and STs) of shape $\lm$. 
Therefore, we have
\begin{equation}
  \label{eq:SSYT->P}
\sum_{\pi\in\SSYT(\lm)} q^{|\pi|} = \sum_{\pi\in\PP(P_\lm,\omega_\lm^\SSYT)} q^{|\pi|},
\end{equation}
\begin{equation}
  \label{eq:RPP->P}
\sum_{\pi\in\RPP(\lm)} q^{|\pi|} = \sum_{\pi\in\PP(P_\lm,\omega_\lm^\RPP)} q^{|\pi|}
\end{equation}
and
\begin{equation}
  \label{eq:ST->P}
\sum_{\pi\in\ST(\lm)} q^{|\pi|} = \sum_{\pi\in\PP(P_\lm,\omega_\lm^\ST)} q^{|\pi|}.  
\end{equation}

It is easy to check that $\pi\in\LL(P_{\delta_{n+2}/\delta_{n}}, \omega_{\delta_{n+2}/\delta_{n}}^\SSYT)$ if and only if $\pi^{-1}\in\Alt_{2n+1}$. 
Thus we have
\begin{equation}
  \label{eq:SSYT->Alt}
\sum_{\pi\in\LL(P_{\delta_{n+2}/\delta_{n}}, \omega_{\delta_{n+2}/\delta_{n}}^\SSYT)} q^{\maj(\pi)} =\sum_{\pi\in\Alt_{2n+1}}q^{\maj(\pi^{-1})} = E_{2n+1}(q).
\end{equation}
Similarly, we have $\pi\in\LL(P_{\delta_{n+2}/\delta_{n}}, \omega_{\delta_{n+2}/\delta_{n}}^\RPP)$ if and only if $\pi=\kappa_{2n+1}\sigma^{-1}$ for some $\sigma\in\Alt_{2n+1}$. Thus,
\begin{equation}
  \label{eq:RPP->Alt}
\sum_{\pi\in\LL(P_{\delta_{n+2}/\delta_{n}}, \omega_{\delta_{n+2}/\delta_{n}}^\RPP)} q^{\maj(\pi)}
=\sum_{\pi\in\Alt_{2n+1}}q^{\maj(\kappa_{2n+1}\pi^{-1})} = E^*_{2n+1}(q).
\end{equation}

By \eqref{eq:lin_ext}, \eqref{eq:SSYT->P} and \eqref{eq:SSYT->Alt}, we have
\begin{equation}
  \label{eq:SSYT->E}
\sum_{\pi\in\SSYT(\delta_{n+2}/\delta_{n})}q^{|\pi|}  = \frac{E_{2n+1}(q)}{(q;q)_{2n+1}}.
\end{equation}
Similarly, by \eqref{eq:lin_ext}, \eqref{eq:RPP->P} and \eqref{eq:RPP->Alt}, we have
\begin{equation}
  \label{eq:RPP->E}
\sum_{\pi\in\RPP(\delta_{n+2}/\delta_n)}q^{|\pi|} = \frac{E^*_{2n+1}(q)}{(q;q)_{2n+1}}.
\end{equation}
Since every ST of shape $\delta_{n+2}/\delta_n$ is obtained from an RPP of shape $\delta_{n+2}/\delta_n$ by adding 1 to each entry in $\delta_{n+2}/\delta_{n+1}$, we have
\begin{equation}
  \label{eq:RPP=ST}
\sum_{\pi\in\ST(\delta_{n+2}/\delta_n)}q^{|\pi|} = \sum_{\pi\in\RPP(\delta_{n+2}/\delta_n)}q^{|\pi|+n+1}.
\end{equation}

\subsection{Excited diagrams and pleasant diagrams}

Excited diagrams are introduced by Ikeda and Naruse\cite{IkedaNaruse09} to describe a combinatorial formula 
for the Schubert classes in the equivariant cohomology ring of the Grassmannians, 
and Naruse \cite{Naruse} derived a hook length formula for standard Young tableaux of a skew shape given in \eqref{eq:naruse}
using equivariant cohomology theory and the excited Young diagrams. In \cite{MPP1}, two $q$-analogs of Naruse's hook length formula are given, 
and one of them, namely \eqref{eq:MPP2}, is defined as a sum over pleasant diagrams. Here, we give definitions of excited diagrams and pleasant diagrams. 

Consider a skew shape $\lm$. Let $D$ be a set of squares in $\lambda$. 
Suppose that $(i,j)\in D$, $(i+1,j), (i,j+1), (i+1,j+1)\not \in D$ and $(i+1,j+1)\in\lambda$.
Then replacing $(i,j)$ by $(i+1,j+1)$ in $D$ is called an \emph{excited move}. An \emph{excited diagram} of $\lm$ 
is a set of squares obtained from $\mu$ by a sequence of excited moves. We denote by $\EE(\lm)$ the set of excited diagrams of $\lm$. 
See Figure~\ref{fig:excited}. 

\begin{figure}[h]
  \centering
\includegraphics[scale=.7]{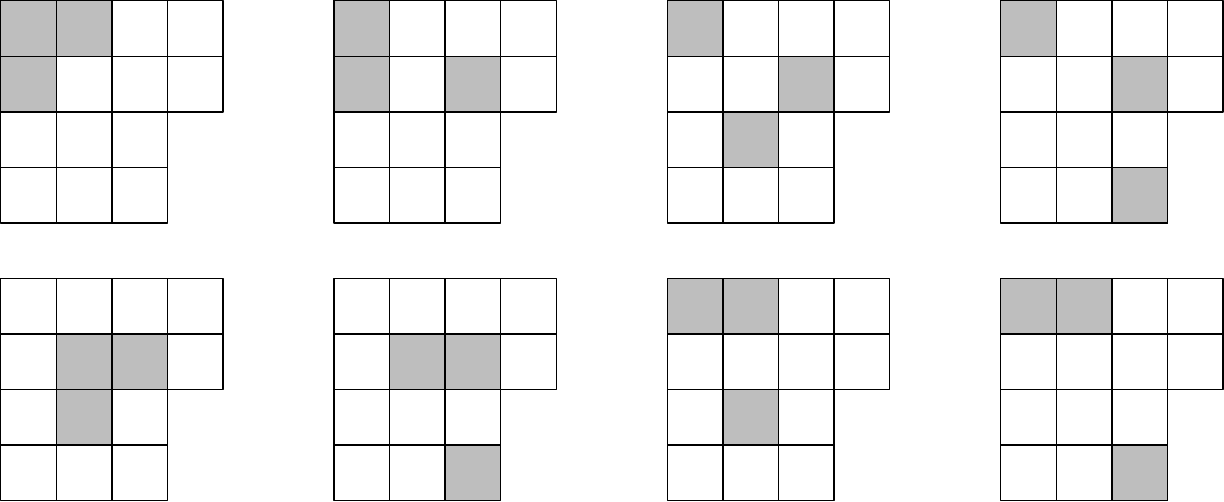}  
  \caption{There are $8$ excited diagrams of $(4,4,3,3)/(2,1)$. The squares in each excited diagram are shaded.}
  \label{fig:excited}
\end{figure}

A \emph{pleasant diagram} of $\lm$ is a subset of $\lambda\setminus D$ for an excited diagram $D$ of $\lm$. 
The set of pleasant diagrams of $\lm$ is denoted by $\PP(\lm)$.

We note that the original definition of a pleasant diagram is different from the definition given here. In \cite[Theorem~6.10]{MPP1} it is shown that the two definitions are equivalent.

\subsection{Dyck paths and Schr\"oder paths}

A \emph{nonnegative lattice path} is a set of points $(x_i,y_i)$ in $\ZZ\times\NN$ for $0\le i\le m$ satisfying
$x_0<x_1<\dots<x_m$. The set $\{(x_i,y_i),(x_{i+1},y_{i+1})\}$ of two consecutive points is called a \emph{step}. 
The step $\{(x_i,y_i),(x_{i+1},y_{i+1})\}$ is called an \emph{up step} (resp.~\emph{down step} and \emph{horizontal step})
if $(x_{i+1},y_{i+1})-(x_{i},y_{i})$ is equal to $(1,1)$ (resp.~$(1,-1)$ and $(2,0)$). 

A \emph{(little) Schr\"oder path} from $(a,0)$ to $(b,0)$ 
is a nonnegative lattice path $\{(x_i,y_i): 0\le i\le m\}$ such that
$(x_0,y_0)=(a,0)$, $(x_m,y_m)=(b,0)$, 
every step is one of an up step, down step and a horizontal step, 
and there is no horizontal step $\{(t,0),(t+2,0)\}$ on the $x$-axis. 
The set of Schr\"order paths from $(a,0)$ to $(b,0)$ is denoted by $\Sch((a,0)\to(b,0))$.
We also denote $\Sch_{2n}=\Sch((-n,0)\to(n,0))$.

A \emph{Dyck path} is a Schr\"oder path without horizontal steps. 
The set of Dyck paths from $(a,0)$ to $(b,0)$ is denoted by $\Dyck((a,0)\to(b,0))$.
We also denote $\Dyck_{2n}=\Dyck((-n,0)\to(n,0))$.

Let $D=\{(x_i,y_i):0\le i\le 2n\}\in\Dyck_{2n}$. Note that $x_i=-n+i$. 
A \emph{valley} (resp.~\emph{peak}) of $D$ is a point $(x_i,y_i)$ such that $0<i<2n$ and $y_{i-1}>y_{i}<y_{i+1}$ (resp.~$y_{i-1}<y_{i}>y_{i+1}$).  
A \emph{high peak} is a peak $(x_i,y_i)$ with $y_i\ge2$.

\section{$q$-Euler numbers and continued fractions}
\label{sec:q-euler}

Prodinger \cite{Prodinger2000} considered the probability $\tau^{\ge\le}_n(q)$ that a random word $w_1\dots w_n$ of positive integers
of length $n$ satisfies the relations $w_1\ge w_2 \le w_3 \ge w_4 \le \cdots$, where each $w_i$ is chosen independently randomly with probability $\mathrm{Pr}(w_i=k)=q^{k-1}(1-q)$ for $0<q<1$. For other choices of inequalities, for example $\ge$ and $<$, the probability $\tau^{\ge <}_n(q)$ is defined similarly. From the definition, one can easily see that
\begin{equation}
  \label{eq:SSYT->tau}
\sum_{\pi\in\SSYT(\delta_{n+2}/\delta_{n})}q^{|\pi|}  = \frac{\tau^{\ge<}_{2n+1} (q)}{(1-q)^{2n+1}}  ,
\end{equation}
\begin{equation}
  \label{eq:RPP->tau}
\sum_{\pi\in\RPP(\delta_{n+2}/\delta_{n})}q^{|\pi|}  = \frac{\tau^{\ge \le}_{2n+1}(q)}{(1-q)^{2n+1}}
\end{equation}
and
\begin{equation}
  \label{eq:ST->tau}
\sum_{\pi\in\ST(\delta_{n+2}/\delta_{n})}q^{|\pi|}  = \frac{\tau^{><}_{2n+1}(q)}{(1-q)^{2n+1}}.
\end{equation}
Observe that, by \eqref{eq:SSYT->E} and \eqref{eq:SSYT->tau}, 
\begin{equation}
  \label{eq:E->tau}
\frac{E_{2n+1}(q)}{(q;q)_{2n+1}} = \frac{\tau^{\ge<}_{2n+1} (q)}{(1-q)^{2n+1}}.
\end{equation}
Similarly, by \eqref{eq:RPP->E} and \eqref{eq:RPP->tau}, 
\begin{equation}
  \label{eq:E*->tau}
\frac{E^*_{2n+1}(q)}{(q;q)_{2n+1}} = \frac{\tau^{\ge \le}_{2n+1}(q)}{(1-q)^{2n+1}}.  
\end{equation}

In this section we show \eqref{eq:MPP_Euler} and \eqref{eq:MPP_Euler*} using Prodinger's results. Prodinger \cite{Prodinger2008} found a continued fraction expression for the generating functions of $\tau^{\ge<}_{2n+1}(q)$ and $\tau^{\ge\le}_{2n+1}(q)$. Using Flajolet's theory \cite{Flajolet1980} of continued fractions we show that
\eqref{eq:MPP_Euler} is equivalent to Prodinger's continued fraction. We prove \eqref{eq:MPP_Euler*} in a similar fashion. However, unlike \eqref{eq:MPP_Euler}, the weight of a Dyck path in \eqref{eq:MPP_Euler*} is not a usual weight used in Flajolet's theory. To remedy this we first express $E^*_{2n+1}(q)$ as a generating function for weighted Schr\"oder paths
and change it to a generating function of weighted Dyck paths.

We recall Flajolet's theory\cite{Flajolet1980} which gives a combinatorial interpretation for the 
continued fraction expansion as a generating function of weighted Dyck paths. 

Let $u=(u_0,u_1,\dots)$, $d=(d_1,d_2,\dots)$ and $w=(w_0,w_1,\dots)$ be sequences satisfying
$w_i=u_id_{i+1}$ for $i\ge0$. For a Dyck path $P\in\Dyck_{2n}$, we define the weight $\wt_w(P)$ with respect to $w$ to be
the product of the weight of each step in $P$, where the weight of an up step $\{(i,j),(i+1,j+1)\}$ is $u_j$ and
the weight of a down step $\{(i,j),(i+1,j-1)\}$ is $d_j$. 
Flajolet \cite{Flajolet1980} showed that the generating function for weighted Dyck paths has a continued fraction expansion:
\begin{equation}
  \label{eq:flajolet}
\sum_{n\ge0}\sum_{P\in\Dyck_{2n}} \wt_w(P) x^{2n} = \cfrac{1}{1-\cfrac{w_0 x^2}{1-\cfrac{w_1 x^2}{1-\cfrac{w_2 x^2}{1-\cdots}}}}  .
\end{equation}


\subsection{The $q$-Euler numbers $E_{2n+1}(q)$}


We give a new proof of \eqref{eq:MPP_Euler} using \eqref{eq:flajolet}.

\begin{prop}\cite[Corollary 1.7]{MPP2}\label{prop:qEuler}
We have
\begin{equation}\label{eqn:qEuler}
\frac{E_{2n+1}(q)}{(q;q)_{2n+1}}= \sum_{P\in \Dyck_{2n}}\prod_{(a,b)\in P}\frac{q^b}{1-q^{2b+1}}.
\end{equation}
\end{prop}

\begin{proof}
By the result of Prodinger \cite[Theorem 4.1]{Prodinger2000} (with replacing $z$ by $x/(1-q)$), we have the following  
continued fraction expansion :
 $$\sum_{n\ge 0}E_{2n+1} (q) \frac{x^{2n+1}}{(q;q)_{2n+1}}
 = \frac{x}{1-q}\cdot \cfrac{1}{1-\cfrac{qx^2/(1-q)(1-q^3)}{1-\cfrac{q^3 x^3/(1-q^3)(1-q^5)}{1-\cfrac{q^5 x^2/(1-q^5)(1-q^7)}{1-\cdots}}}}.
 $$
If we set $u_i=d_i=\frac{q^i}{1-q^{2i+1}}$ and $w_i=u_id_{i+1}$, then \eqref{eq:flajolet} implies 
$$\frac{x}{1-q}\sum_{n\ge 0}\left(\sum_{P\in\Dyck_{2n}}\wt_w(P) \right)x^{2n}=\sum_{n\ge 0}E_{2n+1}(q)\frac{x^{2n+1}}{(q;q)_{2n+1}}.$$
By comparing the coefficients of $x^{2n+1}$, we get 
$$\frac{E_{2n+1}(q)}{(q;q)_{2n+1}}=\frac{1}{1-q}\sum_{P\in\Dyck_{2n}}\wt_w(P).$$
If we change the weight of a Dyck path in the right hand side of \eqref{eqn:qEuler} defined on each point to 
the weight on each step, we obtain 
$$\sum_{P\in \Dyck_{2n}}\prod_{(a,b)\in P}\frac{q^b}{1-q^{2b+1}} = \frac{1}{1-q}\sum_{P\in\Dyck_{2n}}\wt_w(P),$$
which finishes the proof.
\end{proof}


By applying a similar argument, we can express the $q$-secant number $E_{2n}(q)$ as a generating function of weighted Dyck paths
whose weights are given in terms of principle specializations $s_{\lm}(q)=s_{\lm}(1,q,q^2,\dots)$ of skew Schur functions.

\begin{prop}\label{prop:q_Euler_even}
Let $w=(w_0,w_1,\dots)$ be the sequence given by
\[
w_i  = (-1)^i \Delta_{i-2}\Delta_{i+1}/\Delta_{i-1}\Delta_i,
\]
where $\Delta_i = 1$ for $i<0$, 
$\Delta_0 = s_{(0)} (q)$, $\Delta_2 = s_{(4)/(0)}(q)$, $\Delta_{2n}=s_{(3n+1,\dots, 2n+2)/(n-1,\dots, 1,0)}(q)$ and $\Delta_1 = s_{(2)}(q)$, $\Delta_3 = s_{(5,4)/(1,0)}(q)$, $\Delta_{2n-1}=s_{(3n-1,\dots, 2n)/(n-1,\dots, 1,0)}(q)$. 
Then
$$
\frac{E_{2n}(q)}{(q;q)_{2n}}=\sum_{P\in \Dyck_{2n}}\wt_w(P).
$$
\end{prop}

\begin{proof}
In \cite[Theorem 2.5]{Huber2010}, it was shown that 
$$\sum_{n=0}^\infty \frac{E_{2n}(q) x^{2n}}{(q;q)_{2n}}=\frac{1}{\cos _q (x)}=\left(  \sum_{n = 0}^\infty (-1)^n\frac{x^{2n}}{(q;q)_{2n}} \right)^{-1}.$$
The following continued fraction expansion of $1/\cos _q (x)$ has been obtained in \cite[Lemma 5.5]{Lascoux1988}:
\begin{align*}
1/\cos _q (x) &= 1+ s_2 (q) x^2 + s_{32/1}(q) x^4 +\cdots + s_{(n+1,\dots, 2)/(n-1,\dots, 1,0)}(q) x^{2n}+\cdots\\
& =\cfrac{1}{1-\cfrac{\Delta_1 x^2}{1+\cfrac{\Delta_2/\Delta_1 x^2}{1-\cfrac{\Delta_3/\Delta_1\Delta_2 x^2}{ \cfrac{\cdots}{1+\cfrac{(-1)^n x^{2}\Delta_n \Delta_{n+3}/\Delta_{n+1}\Delta_{n+2}}{\cdots }}}}}}.
\end{align*}
By \eqref{eq:flajolet}, $1/\cos_q (x)$ can be interpreted as a generating function of weighted Dyck paths. 
This finishes the proof.
\end{proof}


\subsection{The $q$-Euler numbers $E_{2n+1}^*(q)$}


By using Prodinger's result on $E_{2n+1}^\ast(q)$, we give a new proof of \eqref{eq:MPP_Euler*}.

\begin{prop}\cite[Corollary 1.8]{MPP2}
We have
$$
\frac{E^*_{2n+1}(q)}{(q;q)_{2n+1}} 
= \sum_{P\in\Dyck_{2n}} q^{H(P)} \prod_{(a,b)\in P} \frac{1}{1-q^{2b+1}}.
$$
\end{prop}

\begin{proof}
By Prodinger's result \cite[Theorem 2.2]{Prodinger2000}, we have 
$$
\sum_{n=0}^\infty\frac{E_{2n+1}^{\ast}(q)x^{2n+1}}{(q;q)_{2n+1}}= \sum_{n=0}^\infty\frac{(-1)^n q^{n(n+1)}x^{2n+1}}{(q;q)_{2n+1}}\left/ \sum_{n=0}^\infty\frac{(-1)^n q^{n(n-1)}x^{2n}}{(q;q)_{2n}}\right. .
$$
Let us denote the right hand side by 
\begin{align*}
&  \sum_{n=0}^\infty\frac{(-1)^n q^{n(n+1)}x^{2n+1}}{(q;q)_{2n+1}}\left/ \sum_{n=0}^\infty\frac{(-1)^n q^{n(n-1)}x^{2n}}{(q;q)_{2n}}\right. \\
&= \frac{x}{1-q}\cdot \frac{1}{  \displaystyle \sum_{n=0}^\infty\dfrac{(-1)^n q^{n(n-1)}x^{2n}}{(q;q)_{2n}}\left/\sum_{n=0}^\infty\frac{(-1)^n q^{n(n+1)}x^{2n}}{(q^2;q)_{2n}}\right. .}
\end{align*}
We use Euler's approach \cite{Euler1788} of using Euclid's algorithm to obtain a continued fraction expansion, namely, we apply
\begin{equation}\label{eqn:cf_exp}
\frac{N}{D}=1+a +\frac{N-(1+a)D}{D}
\end{equation}
iteratively (cf. \cite{Bhatnagar}. Note that Bhatnagar used this method to prove number of continued fractions due to Ramanujan).
By applying it once with $a=0$, we have 
\begin{align*}
\dfrac{\displaystyle   \sum_{n=0}^\infty\dfrac{(-1)^n q^{n(n-1)}x^{2n}}{(q;q)_{2n}}}{\displaystyle \sum_{n=0}^\infty\dfrac{(-1)^n q^{n(n+1)}x^{2n}}{(q^2;q)_{2n}}}
&= 1+\frac{ \displaystyle \sum_{n=0}^\infty\dfrac{(-1)^n q^{n(n-1)}x^{2n}}{(q;q)_{2n}}-\sum_{n=0}^\infty\dfrac{(-1)^n q^{n(n+1)}x^{2n}}{(q^2;q)_{2n}}}{\displaystyle \sum_{n=0}^\infty\dfrac{(-1)^n q^{n(n+1)}x^{2n}}{(q^2;q)_{2n}}}\\
&= 1+\frac{\displaystyle  \sum_{n=0}^\infty\dfrac{(-1)^n q^{n(n-1)}(1-q^{2n})x^{2n}}{(q;q)_{2n+1}}}{\displaystyle  \sum_{n=0}^\infty\dfrac{(-1)^n q^{n(n+1)}x^{2n}}{(q^2;q)_{2n}}}\\
&= 1-\frac{\dfrac{x^2}{(1-q)(1-q^3)}}{\displaystyle  \sum_{n=0}^\infty\dfrac{(-1)^n q^{n(n+1)}x^{2n}}{(q^2;q)_{2n}}\left/  \sum_{n=0}^\infty\dfrac{(-1)^n q^{n(n+1)}(1-q^{2n+2})x^{2n}}{(1-q^2)(q^4;q)_{2n}}\right.}.
\end{align*}
After applying \eqref{eqn:cf_exp} $k-1$ times (with appropriate $a$'s), we would have 
$$\sum_{n=0}^\infty\frac{(-1)^n q^{n(n+1)}(q^{2n+2};q^2)_{k-1}x^{2n}}{(q^{2};q^2)_{k-1}(q^{2k};q)_{2n}}
\left/  \sum_{n=0}^\infty\frac{(-1)^n q^{n(n+1)}(q^{2n+2};q^2)_{k}x^{2n}}{(q^2;q^2)_{k}(q^{2k+2};q)_{2n}}
\right.
$$ 
in the last denominator. If we apply \eqref{eqn:cf_exp} one more time with $a=\frac{x^2}{1-q^{2k+1}}$, then we get 
\begin{align*}
&\frac{\displaystyle \sum_{n=0}^\infty\dfrac{(-1)^n q^{n(n+1)}(q^{2n+2};q^2)_{k-1}x^{2n}}{(q^2;q^2)_{k-1}(q^{2k};q)_{2n}}}{\displaystyle \sum_{n=0}^\infty\dfrac{(-1)^n q^{n(n+1)}(q^{2n+2};q^2)_{k}x^{2n}}{(q^2;q^2)_{k}(q^{2k+2};q)_{2n}}}\\
&= 1+\frac{x^2}{1-q^{2k+1}} +\frac{\displaystyle \sum_{n=1}^\infty\dfrac{(-1)^n q^{n(n-1)}(q^{2n};q^2)_{k+1}x^{2n}}{(q^2;q^2)_{k-1}(q^{2k};q)_{2n+2}}}{\displaystyle \sum_{n=0}^\infty\dfrac{(-1)^n q^{n(n+1)}(q^{2n+2};q^2)_{k}x^{2n}}{(q^2;q^2)_{k}(q^{2k+2};q)_{2n}}}\\
&= 1+\dfrac{x^2}{1-q^{2k+1}}-\frac{\dfrac{x^2}{(1-q^{2k+1})(1-q^{2k+3})}}{\displaystyle \sum_{n=0}^\infty\dfrac{(-1)^n q^{n(n+1)}(q^{2n+2};q^2)_{k}x^{2n}}{(q^2;q^2)_{k}(q^{2k+2};q)_{2n}}\left/ \sum_{n=0}^\infty\dfrac{(-1)^n q^{n(n+1)}(q^{2n+2};q^2)_{k+1}x^{2n}}{(q^2;q^2)_{k+1}(q^{2k+4};q)_{2n}}\right.}.
\end{align*}

Eventually we obtain 
\begin{equation}\label{eqn:cf_E*}
\sum_{n=0}^\infty\frac{E_{2n+1}^{\ast}(q)x^{2n+1}}{(q;q)_{2n+1}} = \frac{x}{1-q} \cdot \cfrac{1}{1-\cfrac{\dfrac{x^2}{(1-q)(1-q^3)}}{1+\dfrac{x^2}{1-q^3}-\cfrac{\dfrac{x^2}{(1-q^3)(1-q^5)}}{1+\dfrac{x^2}{1-q^5}-\cfrac{\dfrac{x^2}{(1-q^5)(1-q^7)}}{\cdots}}}}.
\end{equation}
Similarly to \eqref{eq:flajolet}, the right hand side of \eqref{eqn:cf_E*} can be interpreted as a generating function of weighted little Schr\"{o}der paths with the following weights on each step :
$$\begin{cases}
a_i = 1 &\text{ for the up steps going from the level $y=i$ to the level $y=i+1$}\\
c_i = \frac{-1}{1-q^{2i+1}} & \text{ for the horizontal steps on the level $y=i$}\\
b_ i = \frac{1}{(1-q^{2i-1})(1-q^{2i+1})}& \text{ for the down steps going from the level $y=i$ to the level $y=i-1$}.
\end{cases}$$
Note that $c_0 = 0$, since we consider little Schr\"{o}der paths. To change the weighted little Schr\"{o}der paths to weighted Dyck paths, 
we combine the weights of a horizontal step and a pair of up and down steps, and define it as a weight of a pair of up and down steps.
In other words, we add the weights defined on the steps 
$$\includegraphics{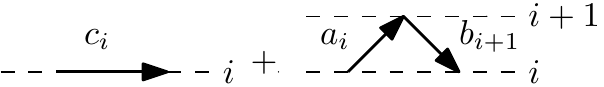}$$
 which gives 
$$\frac{-1}{1-q^{2i+1}}+\frac{1}{(1-q^{2i+1})(1-q^{2i+3})} = \frac{q^{2i+3}}{(1-q^{2i+1})(1-q^{2i+3})}$$
and we define it as a weight defined on a pair of up and down steps
$$\includegraphics{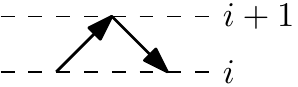}.$$
 This procedure would change the weights defined on little Schr\"{o}der paths to weights on Dyck path, keeping the weights 
 on the up and down steps. Then, considering the extra factor of $1/(1-q)$ in front of the right hand side of \eqref{eqn:cf_E*},
 it is not hard to see that the weights defined on the Dyck path are consistent with the weights used in the right hand side of 
  \eqref{eq:MPP_Euler*}.
\end{proof}

We remark that by applying Flajolet's theory, we can derive the continued fraction expansion which gives the weighted Dyck paths generating function :
$$
\sum_{n=0}^\infty\frac{E_{2n+1}^{\ast}(q)x^{2n+1}}{(q;q)_{2n+1}}=\frac{x}{1-q}\cdot 
\cfrac{1}{1-\cfrac{w_0 x^2}{1-\cfrac{w_1 x^2}{1-\cfrac{w_2 x^2}{1-\cdots}}}},
$$
where 
\begin{equation}\label{eqn:w_i}
w_i =\begin{cases}
\dfrac{q^i}{(1-q^{2i+1})(1-q^{2i+3})}, \text{ if  $i$ is even,} \\
\dfrac{q^{3i+2}}{(1-q^{2i+1})(1-q^{2i+3})}, \text{ if $i$ is odd}.
\end{cases}
\end{equation}
Note that this continued fraction expansion was conjectured in \cite[Conjecture 4.5]{Prodinger2000} (and proved later in \cite{Fulmek2000,Prodinger2008}),
but it can be easily obtained by applying Euler's approach with using the identity
$$\frac{N}{D}=1+\frac{N - D}{D}$$
iteratively.
Then, \eqref{eq:flajolet} implies that 
$$\sum_{n\ge 0}E_{2n+1}^{\ast}(q)\frac{x^{2n+1}}{(q;q)_{2n+1}}= \frac{x}{1-q}\sum_{n\ge 0}\left(\sum_{P\in\Dyck_{2n}}\wt_w(P) \right)x^{2n},$$
where $w=(w_0,w_1,\dots)$ is the sequence defined in \eqref{eqn:w_i}.

\begin{cor}
We have
$$
\sum_{P\in \Dyck_{2n}}q^{H(P)}\prod_{(a,b)\in P}\frac{1}{1-q^{2b+1}}= \frac{1}{1-q}\sum_{P\in\Dyck_{2n}}\wt_w(P),
$$
where $w=(w_0,w_1,\dots)$ is the sequence defined in \eqref{eqn:w_i}.
\end{cor}


\subsection{Other continued fractions}


In this subsection we find continued fraction expressions for various $q$-tangent numbers. 
For each row $\tau^{\alpha\beta}_{2n+1}$, $\frac{(A,B)}{(C,D)}$ and $w_i$ in  Table \ref{tab:cfs}, we assume the following form of continued fraction expansion 
$$
\sum_{n=0}^\infty\frac{\tau^{\alpha\beta}_{2n+1} (q)}{(1-q)^{2n+1}} \cdot \frac{x^{2n+1}}{(q;q)_{2n+1}}=\frac{\displaystyle\sum_{n=0}^{\infty}\dfrac{(-1)^n q^{A n^2 +Bn}x^{2n+1}}{(q;q)_{2n+1}}}{\displaystyle\sum_{n=0}^{\infty}\dfrac{(-1)^n q^{Cn^2+Dn}x^{2n}}{(q;q)_{2n}}}=\frac{x}{1-q}\cdot 
\cfrac{1}{1-\cfrac{w_0 x^2}{1-\cfrac{w_1 x^2}{1-\cfrac{w_2 x^2}{1-\cdots}}}}.
$$
The proof for the continued fraction expansion can be done by applying Euler's approach of using Euclid's algorithm and we omit the details.

\begin{table}[h]
\renewcommand{\arraystretch}{2}
\begin{tabular}{ccc}
\hline
$\tau^{\alpha\beta}_{2n+1} (q)$& $\frac{(A,B)}{(C,D)}$& $w_i$\\
\hline
$\tau^{\ge<}_{2n+1}(q)$ \text{ or } $\tau^{\le>}_{2n+1}(q)$ & $\frac{(0,0)}{(0,0)}$ \text{ or }  $\frac{(2,1)}{(2,-1)}$ &  $\dfrac{q^{2i+1}}{(1-q^{2i+1})(1-q^{2i+3})}$\vspace{2mm}\\
$\tau^{\ge\le}_{2n+1}(q)$ & $\frac{(1,1)}{(1,-1)}$ & $\begin{cases}
\dfrac{q^i}{(1-q^{2i+1})(1-q^{2i+3})}, \text{ if  $i$ is even,}  \vspace{2mm}\\
\dfrac{q^{3i+2}}{(1-q^{2i+1})(1-q^{2i+3})}, \text{ if $i$ is odd}.
\end{cases}$\\
$\tau^{<>}_{2n+1}(q)$ & $\frac{(1,1)}{(1,0)}$ & $-$\\
$\tau^{><}_{2n+1}(q)$ & $\frac{(1,0)}{(1,0)}$ & $\begin{cases}
\dfrac{q^{3i+2}}{(1-q^{2i+1})(1-q^{2i+3})}, \text{ if  $i$ is even,} \vspace{2mm}\\
\dfrac{q^{i}}{(1-q^{2i+1})(1-q^{2i+3})}, \text{ if $i$ is odd}.
\end{cases}$ \vspace{2mm}\\
$-$ & $\frac{(0,1)}{(0,1)}$ &  $\dfrac{q^{2i+2}}{(1-q^{2i+1})(1-q^{2i+3})}$ \vspace{2mm}\\
$-$ & $\frac{(2,0)}{(2,-2)}$ & $\dfrac{q^{2i}}{(1-q^{2i+1})(1-q^{2i+3})}$\\
$\tau^{\le\ge}_{2n+1}(q)$ & $\frac{(1,0)}{(1,-1)}$ & $-$\\
\hline
\end{tabular}
\caption{Continued fraction expressions for the generating functions of various $q$-tangent numbers.}\label{tab:cfs}
\end{table}

\section{Prodinger's $q$-Euler numbers and Foata-type bijections}
\label{sec:foata}

Recall from the previous section that we have
\begin{equation}
  \label{eq:1}
\frac{\tau^{\ge<}_{2n+1} (q)}{(1-q)^{2n+1}}    = \sum_{\pi\in\SSYT(\delta_{n+2}/\delta_{n})}q^{|\pi|}  = 
\frac{\sum_{\pi\in\Alt_{2n+1}} q^{\maj(\pi^{-1})} }{(q;q)_{2n+1}}
\end{equation}
and
\begin{equation}
  \label{eq:2}
\frac{\tau^{\ge \le}_{2n+1}(q)}{(1-q)^{2n+1}} = \sum_{\pi\in\RPP(\delta_{n+2}/\delta_{n})}q^{|\pi|}  = 
\frac{\sum_{\pi\in\Alt_{2n+1}} q^{\maj(\kappa_{2n+1}\pi^{-1})}}{(q;q)_{2n+1}}.
\end{equation}
Using recurrence relations and generating functions, Huber and Yee \cite{Huber2010} showed that
\begin{equation}
  \label{eq:3}
\frac{\tau^{\ge<}_{2n+1} (q)}{(1-q)^{2n+1}}  =
\frac{\sum_{\pi\in\Alt_{2n+1}} q^{\inv(\pi^{-1})} }{(q;q)_{2n+1}}
\end{equation}
and
\begin{equation}
  \label{eq:4}
\frac{\tau^{\ge\le}_{2n+1} (q)}{(1-q)^{2n+1}}  =
\frac{\sum_{\pi\in\Alt_{2n+1}} q^{\inv(\pi^{-1})-\ndes(\pi_e)} }{(q;q)_{2n+1}}.
\end{equation}
Observe that \eqref{eq:2} and \eqref{eq:4} imply
\begin{equation}
  \label{eq:5}
\sum_{\pi\in\Alt_{2n+1}} q^{\maj(\kappa_{2n+1}\pi^{-1})}
=\sum_{\pi\in\Alt_{2n+1}} q^{\inv(\pi^{-1})-\ndes(\pi_e)}.
\end{equation}

In Section~\ref{sec:prodingers-q-euler} we explore various versions of $q$-Euler numbers considered by Prodinger \cite{Prodinger2000} 
and show that they have similar identities as \eqref{eq:1}, \eqref{eq:2}, \eqref{eq:3} and \eqref{eq:4}. In Section~\ref{sec:foata-type-bijection} we give a bijective proof of \eqref{eq:5}
by finding a Foata-type bijection. In Section~\ref{sec:more_foata} we give various Foata-type bijections which imply similar identities as \eqref{eq:5}. 

\subsection{Prodinger's $q$-Euler numbers}
\label{sec:prodingers-q-euler}

 Prodinger \cite{Prodinger2000} showed that the generating function for $\tau^{\alpha\beta}_{n} (q)$ for any choice of alternating inequalities $\alpha$ and $\beta$, i.e.,
\[
(\alpha,\beta)\in \{(\ge,\le)\, (\ge,<),(>,\le),(>,<),(\le,\ge),(\le,>),(<,\ge),(<,>)\},
\]
has a nice expression as a quotient of series. Observe that, by reversing a word of length $2n+1$, we have $\tau_{2n+1}^{\ge<}(q) = \tau_{2n+1}^{>\le}(q)$ and
$\tau_{2n+1}^{\le>}(q) = \tau_{2n+1}^{<\ge}(q)$. Also, by reversing a word of length $2n$, we have
$\tau_{2n}^{\ge\le}(q) = \tau_{2n}^{\le\ge}(q)$,
$\tau_{2n}^{\ge<}(q) = \tau_{2n}^{\le>}(q)$,
$\tau_{2n}^{>\le}(q) = \tau_{2n}^{<\ge}(q)$ and
$\tau_{2n}^{><}(q) = \tau_{2n}^{<>}(q)$.
Therefore, we only need to consider  $6$ $q$-tangent numbers $\tau^{\alpha\beta}_{2n+1}$ and $4$ $q$-secant numbers $\tau^{\alpha\beta}_{2n}$. 

The following lemma is straightforward to check.
\begin{lem}\label{lem:reverse_complement}
For $\pi=\pi_1\pi_2\dots\pi_n\in \Sym_n$, let 
\[
\phi(\pi) = (n+1-\pi_n)(n+1-\pi_{n-1})\cdots(n+1-\pi_1).
\]
Then the map $\phi$ induces a bijection $\phi:\Alt_{2n+1}\to\Ralt_{2n+1}$ 
and a bijection $\phi:\Alt_{2n}\to\Alt_{2n}$. Moreover, for $\pi\in\Alt_{2n+1}$ and $\sigma\in\Alt_{2n}$, we have
\[
(\inv(\phi(\pi)),\des(\phi(\pi)_o),\des(\phi(\pi)_e)) = 
(\inv(\pi),\des(\pi_o),\des(\pi_e))
\]
and
\[
(\inv(\phi(\sigma)),\des(\phi(\sigma)_o)) = (\inv(\sigma),\des(\sigma_e)).
\]
\end{lem}

Now we state a unifying theorem for Prodinger's $q$-tangent numbers. 

\begin{thm}\label{thm:q-tan}
For each row $\tau_{2n+1}^{\alpha\beta}(q), \TAB, M, I, (A,B)/(C,D)$ in Table~\ref{tab:q-tan}, we have
\[
f_{2n+1}:=\frac{\tau_{2n+1}^{\alpha\beta}(q)}{(1-q)^{2n+1}}=\sum_{\pi\in \TAB}q^{|\pi|} = \frac{M}{(q;q)_{2n+1}} = \frac{I}{(q;q)_{2n+1}} ,
\]
whose generating function is
\[
\sum_{n\ge0} f_{2n+1} x^{2n+1} = 
\frac{\sum_{n\ge0} (-1)^nq^{An^2+Bn}x^{2n+1}/(q;q)_{2n+1}}
{\sum_{n\ge0}(-1)^nq^{Cn^2+Dn}x^{2n}/(q;q)_{2n}}.
\]
\end{thm}
\begin{proof}
This is obtained by combining known results. The connection between $\tau^{\alpha\beta}_{2n+1}(q)$ and $\TAB$
is obvious from their definitions. The connection between $\TAB$ and $M$ follows from the $P$-partition theory \eqref{eq:lin_ext}. 
The connection between $\tau^{\alpha\beta}_{2n+1}(q)$ and $(A,B)/(C,D)$ is due to Prodinger \cite{Prodinger2000}. The connection between $I$ and $(A,B)/(C,D)$ is due to Huber and Yee \cite{Huber2010}. See Figure~\ref{fig:diagram} for an illustration of these connections. 
We note that Huber and Yee \cite{Huber2010} only considered $\Alt_{2n+1}$ for $I$, but Lemma~\ref{lem:reverse_complement} implies that $\Alt_{2n+1}$ can be replaced by $\Ralt_{2n+1}$. 
\end{proof}

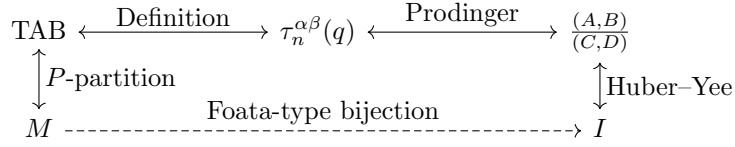
\begin{figure}
\[
\begin{tikzcd}
\tikzcdset{row sep=10cm}
\TAB \arrow[leftrightarrow]{rrr}{\mbox{Definition}} \arrow[leftrightarrow]{d}{\mbox{$\displaystyle P$-partition}}
&&& \tau^{\alpha\beta}_n(q) \arrow[leftrightarrow]{rrr}{\mbox{Prodinger}} &&& \frac{(A,B)}{(C,D)} \arrow[leftrightarrow]{d}{\mbox{Huber--Yee}} \\
M \arrow[dashrightarrow]{rrrrrr}{\mbox{Foata-type bijection}} &&&&&& I
\end{tikzcd}
\]
  \caption{The connections in Theorems~\ref{thm:q-tan} and \ref{thm:q-sec}.}
  \label{fig:diagram}
\end{figure}

\begin{table}
\renewcommand{\arraystretch}{2}
  \centering
  \begin{tabular}{c l l l c}
\hline
$\tau_{2n+1}^{\alpha\beta}(q)$ & \qquad$\TAB$ & \qquad\qquad$M$ & \qquad\qquad$I$ & $\frac{(A,B)}{(C,D)}$\\ \hline
$\tau^{\ge<}_{2n+1}(q)$ & $\SSYT(\delta_{n+2}/\delta_n)$ & $\displaystyle\sum_{\pi\in\Alt_{2n+1}}q^{\maj(\pi^{-1})}$ &
$\displaystyle \sum_{\pi\in\Alt^*_{2n+1}}q^{\inv(\pi)}$ & $\frac{(0,0)}{(0,0)}$ \\
$\tau^{\ge\le}_{2n+1}(q)$ & $\RPP(\delta_{n+2}/\delta_n)$ & $\displaystyle\sum_{\pi\in\Alt_{2n+1}}q^{\maj(\kappa_{2n+1}\pi^{-1})}$
 & $\displaystyle \sum_{\pi\in\Alt^*_{2n+1}}q^{\inv(\pi)-\ndes(\pi_e)}$ & $\frac{(1,1)}{(1,-1)}$ \\
$\tau^{><}_{2n+1}(q)$ & $\ST(\delta_{n+2}/\delta_n)$ & $\displaystyle\sum_{\pi\in\Alt_{2n+1}}q^{\maj(\eta_{2n+1}\pi^{-1})}$
 & $\displaystyle \sum_{\pi\in\Alt^*_{2n+1}}q^{\inv(\pi)+\nasc(\pi_e)}$ & $\frac{(1,0)}{(1,0)}$ \\
$\tau^{<\ge}_{2n+1}(q)$ & $\SSYT(\delta_{n+3}^{(1,1)}/\delta_{n+1})$ & $\displaystyle\sum_{\pi\in\Ralt_{2n+1}}q^{\maj(\pi^{-1})}$&
$\displaystyle \sum_{\pi\in\Alt^*_{2n+1}}q^{\inv(\pi)}$ & $\frac{(0,0)}{(0,0)}$ \\
$\tau^{\le\ge}_{2n+1}(q)$ & $\RPP(\delta_{n+3}^{(1,1)}/\delta_{n+1})$ & $\displaystyle\sum_{\pi\in\Ralt_{2n+1}}q^{\maj(\eta_{2n+1}\pi^{-1})}$
 & $\displaystyle \sum_{\pi\in\Alt^*_{2n+1}}q^{\inv(\pi)-\asc(\pi_o)}$ & $\frac{(1,0)}{(1,-1)}$ \\
$\tau^{<>}_{2n+1}(q)$ & $\ST(\delta_{n+3}^{(1,1)}/\delta_{n+1})$ & $\displaystyle\sum_{\pi\in\Ralt_{2n+1}}q^{\maj(\kappa_{2n+1}\pi^{-1})}$
 & $\displaystyle \sum_{\pi\in\Alt^*_{2n+1}}q^{\inv(\pi)+\des(\pi_o)}$ & $\frac{(1,1)}{(1,0)}$ \\
\hline
  \end{tabular}
  \caption{Interpretations for Prodinger's $q$-tangent numbers. The notation $\Alt^*_{2n+1}$ means it can be either $\Alt_{2n+1}$ or $\Ralt_{2n+1}$.}
  \label{tab:q-tan}
\end{table}

By the same arguments, we obtain a unifying theorem for Prodinger's $q$-secant numbers. Note that we consider $6$ $q$-secant numbers even though two pairs of them are equal because their natural combinatorial interpretations are different. 

\begin{thm}\label{thm:q-sec}
For each row $\tau_{2n}^{\alpha\beta}(q), \TAB, M, I, 1/(C,D)$ in Table~\ref{tab:q-sec}, we have
\[
f_{2n}: = \frac{\tau_{2n}^{\alpha\beta}(q)}{(1-q)^{2n}} = \sum_{\pi\in \TAB}q^{|\pi|} = \frac{M}{(q;q)_{2n}} = \frac{I}{(q;q)_{2n}},
\]
whose generating function is
\[
\sum_{n\ge0} f_{2n} x^{2n} = 
\frac{1}{\sum_{n\ge0}(-1)^nq^{Cn^2+Dn}x^{2n}/(q;q)_{2n}}.
\]
\end{thm}

\begin{table}
\renewcommand{\arraystretch}{2}
  \centering
  \begin{tabular}{c l l l c}
\hline
$\tau_{2n}^{\alpha\beta}(q)$ & \qquad $\TAB$ & \qquad \qquad$M$ & \qquad\qquad $I$ & $\frac{1}{(C,D)}$\\ \hline
$\tau^{\ge<}_{2n}(q)$ & $\SSYT(\delta^{(0,1)}_{n+2}/\delta_n)$ & $\displaystyle\sum_{\pi\in\Alt_{2n}}q^{\maj(\pi^{-1})}$ &
$\displaystyle \sum_{\pi\in\Alt_{2n}}q^{\inv(\pi)}$ & $\frac{1}{(0,0)}$ \\
$\tau^{\ge\le}_{2n}(q)$ & $\RPP(\delta^{(0,1)}_{n+2}/\delta_n)$ & $\displaystyle\sum_{\pi\in\Alt_{2n}}q^{\maj(\kappa_{2n}\pi^{-1})}$
 & $\displaystyle \sum_{\pi\in\Alt_{2n}}q^{\inv(\pi)-\asc(\pi_*)}$ & $\frac{1}{(1,-1)}$ \\
$\tau^{><}_{2n}(q)$ & $\ST(\delta^{(0,1)}_{n+2}/\delta_n)$ & $\displaystyle\sum_{\pi\in\Alt_{2n}}q^{\maj(\eta_{2n}\pi^{-1})}$
 & $\displaystyle \sum_{\pi\in\Alt_{2n}}q^{\inv(\pi)+\nasc(\pi_*)}$ & $\frac{1}{(1,0)}$ \\
$\tau^{<\ge}_{2n}(q)$ & $\SSYT(\delta_{n+2}^{(1,0)}/\delta_{n})$ & $\displaystyle\sum_{\pi\in\Ralt_{2n}}q^{\maj(\pi^{-1})}$&
$\displaystyle \sum_{\pi\in\Ralt_{2n}}q^{\inv(\pi)}$ & $\frac{1}{(2,-1)}$ \\
$\tau^{\le\ge}_{2n}(q)$ & $\RPP(\delta_{n+2}^{(1,0)}/\delta_{n})$ & $\displaystyle\sum_{\pi\in\Ralt_{2n}}q^{\maj(\eta_{2n}\pi^{-1})}$
 & $\displaystyle \sum_{\pi\in\Ralt_{2n}}q^{\inv(\pi)-\ndes(\pi_*)}$ & $\frac{1}{(1,-1)}$ \\
$\tau^{<>}_{2n}(q)$ & $\ST(\delta_{n+2}^{(1,0)}/\delta_{n})$ & $\displaystyle\sum_{\pi\in\Ralt_{2n}}q^{\maj(\kappa_{2n}\pi^{-1})}$
 & $\displaystyle \sum_{\pi\in\Ralt_{2n}}q^{\inv(\pi)+\des(\pi_*)}$ & $\frac{1}{(1,0)}$ \\
\hline
  \end{tabular}
  \caption{Interpretations for Prodinger's $q$-secant numbers. The notation $\pi_*$ means it can be either $\pi_o$ or $\pi_e$.}
  \label{tab:q-sec}
\end{table}

It is natural to ask a direct bijective proof of $M=I$ in Theorems~\ref{thm:q-tan} and \ref{thm:q-sec}. For the next two subsections we accomplish this by finding Foata-type bijections. 

\begin{remark}
The last two expressions for $I$ in Table~\ref{tab:q-sec} have not been considered in \cite{Huber2010}. Therefore, in order to complete the proof of Theorem~\ref{thm:q-sec} we need to find a connection between these expressions with other known results. 
In Section~\ref{sec:more_foata} we find a bijective proof of $M=I$ for all cases in Tables~\ref{tab:q-tan} and \ref{tab:q-sec}. Note that in Table~\ref{tab:q-sec} the second and fifth rows are equal and the third and sixth rows are equal, which imply
\[
\sum_{\pi\in\Alt_{2n}}q^{\inv(\pi)-\asc(\pi_*)} = \sum_{\pi\in\Ralt_{2n}}q^{\inv(\pi)-\ndes(\pi_*)}
\]
and
\[
\sum_{\pi\in\Alt_{2n}}q^{\inv(\pi)+\nasc(\pi_*)} = \sum_{\pi\in\Ralt_{2n}}q^{\inv(\pi)+\des(\pi_*)}.
\]
It would be interesting to find a direct bijective proof of the above two identities.
\end{remark}

\subsection{Foata-type bijection for $E^*_{2n+1}(q)$.}\
\label{sec:foata-type-bijection}

We denote by $\Alt^{-1}_n$ the set of permutations $\pi\in \Sym_n$ with $\pi^{-1}\in \Alt_n$.
Note that $\pi\in\Alt^{-1}_n$ if and only if the relative position of $2i-1$ and $2i$ in $\pi$,
for all $1\le i\le \flr{n/2}$, is
\[
\pi= \cdots (2i-1) \cdots (2i) \cdots
\]
and the relative position of $2i$ and $2i+1$ in $\pi$, for all $1\le i\le \flr{(n-1)/2}$, is
\[
\pi= \cdots (2i+1) \cdots (2i) \cdots.
\]

Let $\prec$ be a total order on $\NN$.
For a word $w_1\dots w_k$ consisting of distinct positive integers, we define
$f(w_1\dots w_k,\prec)$ as follows. 
Let $b_0,b_1,\dots,b_m$ be the integers such that 
\begin{itemize}
\item $0=b_0< b_1<\dots<b_m=k-1$,
\item if $w_{k-1}\prec w_k$, then $w_{b_1},\dots,w_{b_m}\prec w_k \prec w_j$ for all $j\in[k-1]\setminus\{b_1,\dots,b_m\}$, and
\item if $w_{k}\prec w_{k-1}$, then $w_j \prec w_k \prec w_{b_1},\dots,w_{b_m} $ for all $j\in[k-1]\setminus\{b_1,\dots,b_m\}$. 
\end{itemize}
For $1\le j\le m$, let $B_j=w_{b_{j-1}+1}\dots w_{b_j}$. 
We denote
\[
B(w_1\dots w_k,\prec) = (B_1,B_2,\dots, B_m).
\]
Note that $w_1\dots w_{k-1}w_k$ is the concatenation $B_1B_2\dots B_m w_k$. 
Let $B_j'=w_{b_j}w_{b_{j-1}+1}\dots w_{b_{j}-1}$. Then we define
\[
f(w_1\dots w_k,\prec)=B'_1B'_2\dots B'_m w_k. 
\]
For example,
\[
f(496318725,<) =f(|4|963|1|872|5,<)= 439612875,
\]
where $|4|963|1|872|5$ means that $B(496318725,<)=(4,963,1,872)$. 

For a permutation $\pi=\pi_1\dots \pi_n\in \Sym_n$ and a total order $\prec$ on $\NN$, we define
$F(\pi,\prec)$ as follows. Let $w^{(1)}=\pi_1$. For $2\le k\le n$, let
$w^{(k)}=f(w^{(k-1)}\pi_k,\prec)$. Finally $F(\pi,\prec)=w^{(n)}$.
Note that for the natural order $1<2<\cdots$, the map $F(\pi,<)$ is the same as the Foata map. 

For $i\ge1$, we define $<_i$ to be the total order on $\NN$ 
obtained from the natural ordering by reversing the order of $i$ and $i+1$, i.e.,
for $a<b$ with $(a,b)\ne(i,i+1)$, we have $a<_ib$ and $i+1<_i i$.

For $\pi\in \Alt^{-1}_{2n+1}$, we define $\FA(\pi)$ as follows. 
First, we set $w^{(1)}=\pi_1$. For $2\le k\le 2n+1$, there are two cases:
\begin{itemize}
\item If $\pi_k=2i$ and $\pi_1\dots\pi_{k-1}$ does not have $2i+2$, then $w^{(k)}=f(w^{(k-1)}\pi_k, <_{2i})$.
\item Otherwise, $w^{(k)}=f(w^{(k-1)}\pi_k, <)$.
\end{itemize}
Then $\FA(\pi)$ is defined to be $w^{(2n+1)}$. 

\begin{exam}
Let $\pi=317295486 \in\Alt^{-1}_9$. Then 
\begin{align*}
w^{(1)} &= 3, \\
w^{(2)} &= f(|3|1,<)=31, \\
w^{(3)} &= f(|3|1|7,<) =317 , \\
w^{(4)} & =f(|\mathbf{3}17|\mathbf{2},<_2)=7312,   &\mbox{(There is no 4 before 2.)}\\
 w^{(5)} &= f(|7|3|1|2|9,<)=73129,\\
 w^{(6)} &= f(|7|3129|5,<)=793125,\\
 w^{(7)} &= f(|793|1|2|\mathbf{5}|\mathbf{4},<_4)=3791254, &\mbox{(There is no 6 before 4.)}\\
 w^{(8)} &= f(|3|7|\mathbf{9}|1|2|5|4|\mathbf{8},<_8)=37912548, &\mbox{(There is no 10 before 8.)}\\
 w^{(9)} &= f(|37|9|12548|6,<)=739812546, &\mbox{(There is 8 before 6.)}
\end{align*}
where the two integers whose order is reversed are written in bold face.
Thus $\FA(317295486)=739812546$. 
\end{exam}

\begin{thm}\label{thm:foata}
The map $\FA$ induces a bijection $\FA:\Alt^{-1}_{2n+1}\to \Alt^{-1}_{2n+1}$. Moreover,
if $\pi\in \Alt^{-1}_{2n+1}$ and $\sigma=\FA(\pi)$, then
\[
\maj(\kappa_{2n+1} \pi) = \inv(\sigma)-\ndes((\sigma^{-1})_e).
\]  
\end{thm}
\begin{proof}
One can easily see that each step in the construction of $\FA$ is invertible. Thus $\FA:\Sym_n\to \Sym_n$ is a bijection. By the construction of $\FA$, the relative positions of consecutive integers never change. Thus $\Des(\pi^{-1})=\Des(\FA(\pi)^{-1})$, which implies that
$\FA:\Alt^{-1}_{2n+1}\to \Alt^{-1}_{2n+1}$ is also a bijection. 
 
For the proof of the second statement, let $\pi=\pi_1\dots\pi_{2n+1}\in\Alt_{2n+1}^{-1}$
and $w^{(k)}$ the word in the construction of $\FA(\pi)$ for $k\in[2n+1]$. 
We claim that for every $k\in[2n+1]$,
\begin{equation}
  \label{eq:claim}
\maj(\kappa\pi^{(k)}) = \inv(w^{(k)})-t(w^{(k)}),
\end{equation}
where $\pi^{(k)}=\pi_1\dots\pi_k$, $\kappa\pi^{(k)}$ is the word obtained from $\pi^{(k)}$ by interchanging $2i$ and $2i+1$
for all $i$ for which $\pi^{(k)}$ contains both $2i$ and $2i+1$, and
$t(w^{(k)})$ is the number of even integers $2i$ for which $w^{(k)}$ contains $2i$ 
and there is no $2i+2$ before $2i$ in $w^{(k)}$. 
Note that if $w$ is a permutation then $t(w)=\ndes(w^{-1}_e)$. 
Thus \eqref{eq:claim} for $k=2n+1$ implies the second statement of the lemma.

We prove \eqref{eq:claim} by induction on $k$. It is true for $k=0$ since both sides are equal to $0$. Assuming the claim \eqref{eq:claim} for $k-1\ge 0$, we consider the claim for $k$. There are several cases as follows.

\begin{description}
\item[Case 1] $\pi_k=2i+1$ for some $i$. Then $w^{(k)}=f(w^{(k-1)}\pi_k, <)$.
Since $\pi\in\Alt_{2n+1}^{-1}$, $\pi^{(k)}$ does not have $2i$ and we have $\maj(\kappa\pi^{(k)})=\maj((\kappa\pi^{(k-1)})\pi_k)$.
\begin{description}
\item[Subcase 1-(a)] $\pi_{k-1}>2i+1$. Since the last element of $\kappa\pi^{(k-1)}$ remains greater than $\pi_k$, we have
$\maj(\kappa\pi^{(k)})=\maj((\kappa\pi^{(k-1)}))+k-1$.
Let 
\[
B(w^{(k-1)}\pi_k,<)=(B_1,\dots,B_\ell).
\]
Since the last element of $w^{(k-1)}$ is $\pi_{k-1}>\pi_k$, for all $j\in[\ell]$, the last element of $B_j$ is the only element in $B_j$ which is greater than $\pi_k$. For every element $x$ in $B_j$, if it is the last element in $B_j$, then $x$ and $\pi_k$ form a new inversion in $w^{(k)}$, and otherwise
$x$ and the last element of $B_j$ form a new inversion in $w^{(k)}$. Thus $\inv(w^{(k)})=\inv(w^{(k-1)})+k-1$. 
Moreover, we have $t(w^{(k)})  = t(w^{(k-1)})$ because the relative position of $2r$ and $2r+2$ does not change in $w^{(k)}$. 
The relative position of $2r$ and $2r+2$ can be changed only if $2r<2i+1<2r+2$, which implies $r=i$. This is a contradiction to the fact that $2i$ is not in $\pi^{(k)}$. Thus $\inv(w^{(k)})-t(w^{(k)})=\inv(w^{(k-1)}) -t(w^{(k-1)})+k-1$ and the claim is also true for $k$. 

\item[Subcase 1-(b)] $\pi_{k-1}<2i$. In this case $\maj(\kappa\pi^{(k)})=\maj((\kappa\pi^{(k-1)}))$. By the same arguments one
can show that $\inv(w^{(k)}) = \inv(w^{(k-1)})$ and $t(w^{(k)})=(w^{(k-1)})$, which imply the claim for $k$.
 
\end{description}
\item[Case 2] $\pi_k=2i$ for some $i$. Since $\pi\in\Alt^{-1}_{2n+1}$, $\pi^{(k)}$ contains $2i+1$. There are $5$ subcases.

  \begin{description}
  \item [Subcase 2-(a)] $\pi_{k-1}=2i+1$. In this case $\pi^{(k)}$ does not have $2i+2$ since $\pi\in\Alt^{-1}_{2n+1}$. 
Thus $w^{(k)}=f(w^{(k-1)}\pi_k,<_{2i})$. 
Observe that the last two elements of $\kappa\pi^{(k)}$ are $2i$ and $2i+1$ in this order, which implies $\maj(\kappa\pi^{(k)}) = \maj(\kappa\pi^{(k-1)})$.
Let 
\[
B(w^{(k-1)}\pi_k,<_{2i})=(B_1,\dots,B_\ell).
\]
The last element of $w^{(k-1)}$ is $\pi_{k-1}=2i+1$, which is smaller than $\pi_k=2i$ with respect to the order $<_{2i}$. 
Hence, for every $j\in[\ell]$, the last element in $B_j$ is the only element there which is smaller than $2i$ with respect to $<_{2i}$. 
If $x\in B_j$ is the last element in $B_j$, then $x<_{2i} 2i$. In this case $x$ and $2i$ form a new inversion in $w^{(k)}$ 
if and only if  $x=2i+1$. On the other hand, if $x\in B_j$ is not the last element $y$ in $B_j$, then $y<_{2i} 2i <_{2i} x$. 
In this case when we construct $w^{(k)}$ we lose the inversion formed by $x$ and $y$ in $w^{(k-1)}$
and get a new inversion formed by $x$ and $2i$. Therefore, in total, $\inv(w^{(k)}) = \inv(w^{(k-1)}) +1$. 
Since the relative position of $2r$ and $2r+2$ does not change, the last element of $w^{(k)}$ is $2i$ and $w^{(k)}$ does not have $2i+2$, we have
$t(w^{(k)})  = t(w^{(k-1)})+1$. Therefore, $\inv(w^{(k)})-t(w^{(k)}) = \inv(w^{(k-1)})-t(w^{(k-1)})$, which implies the claim for $k$.

  \item [Subcase 2-(b)] $\pi_{k-1}>2i+1$ and $\pi^{(k)}$ has $2i+2$. In this case we have
$w^{(k)}=f(w^{(k-1)}\pi_k,<)$. By similar arguments as in Subcase 1-(a), we have
$\maj(\kappa\pi^{(k)})=\maj(\kappa\pi^{(k-1)})+k-1$, 
$\inv(w^{(k)}) = \inv(w^{(k-1)})+k-1$ and
$t(w^{(k)}) = t(w^{(k-1)})$. Thus $\inv(w^{(k)})-t(w^{(k)}) = \inv(w^{(k-1)})-t(w^{(k-1)})+k-1$, which implies the claim for $k$. 

  \item [Subcase 2-(c)] $\pi_{k-1}>2i+1$ and $\pi^{(k)}$ does not have $2i+2$.
In this case, $w^{(k)}=f(w^{(k-1)}\pi_k,<_{2i})$. By similar arguments as in Subcase 2-(a), we obtain
$\maj(\kappa\pi^{(k)})=\maj(\kappa\pi^{(k-1)})+k-1$, 
$\inv(w^{(k)}) = \inv(w^{(k-1)})+k$ and
$t(w^{(k)}) = t(w^{(k-1)})+1$. Thus $\inv(w^{(k)})-t(w^{(k)}) = \inv(w^{(k-1)})-t(w^{(k-1)})+k-1$, which implies the claim for $k$.

  \item [Subcase 2-(d)] $\pi_{k-1}< 2i$ and $\pi^{(k)}$ has $2i+2$.
In this case, $w^{(k)}=f(w^{(k-1)}\pi_k,<)$. By similar arguments as in Subcase 1-(a), we obtain
$\maj(\kappa\pi^{(k)})=\maj(\kappa\pi^{(k-1)})$, 
$\inv(w^{(k)}) = \inv(w^{(k-1)})$ and
$t(w^{(k)}) = t(w^{(k-1)})$. Thus we have $\inv(w^{(k)})-t(w^{(k)}) = \inv(w^{(k-1)})-t(w^{(k-1)})+k-1$, which implies the claim for $k$. 

  \item [Subcase 2-(e)] $\pi_{k-1}< 2i$ and $\pi^{(k)}$ does not have $2i+2$. In this case, we have
 $w^{(k)}=f(w^{(k-1)}\pi_k, <_{2i})$. By similar arguments as in Subcase 2-(a), we obtain
$\maj(\kappa\pi^{(k)})=\maj(\kappa\pi^{(k-1)})$, 
$\inv(w^{(k)}) = \inv(w^{(k-1)})+1$ and
$t(w^{(k)}) = t(w^{(k-1)})-1$. Thus $\inv(w^{(k)})-t(w^{(k)}) = \inv(w^{(k-1)})-t(w^{(k-1)})$, which implies the claim for $k$. 
  \end{description}
\end{description}
In any case, the claim is true for $k$, which completes the proof. 
\end{proof}

\begin{cor}
We have
\[
\sum_{\pi\in \Alt_{2n+1}} q^{\maj(\kappa_{2n+1} \pi^{-1})}
 =\sum_{\pi\in \Alt_{2n+1}} q^{\inv(\pi)-\ndes(\pi_e)}.
\]
\end{cor}
\begin{proof}
By Theorem~\ref{thm:foata},  
\[
\sum_{\pi\in \Alt_{2n+1}} q^{\maj(\kappa_{2n+1} \pi^{-1})} =
\sum_{\pi\in \Alt^{-1}_{2n+1}} q^{\maj(\kappa_{2n+1} \pi)}=
\sum_{\sigma\in \Alt^{-1}_{2n+1}} q^{\inv(\sigma)-\ndes((\sigma^{-1})_e)}.
\]
Since
\[
\sum_{\sigma\in \Alt^{-1}_{2n+1}} q^{\inv(\sigma)-\ndes((\sigma^{-1})_e)}
=\sum_{\sigma\in \Alt_{2n+1}} q^{\inv(\sigma)-\ndes(\sigma_e)},
\]
we are done.
\end{proof}

\subsection{Foata-type bijections for other $q$-Euler numbers}
\label{sec:more_foata}

We give Foata-type bijections which imply the identities $M=I$ in Tables~\ref{tab:q-tan} and \ref{tab:q-sec}.

\begin{thm}\label{thm:foata_mod}
For each row  $\mathrm{Set}$,  $N$, $A(\pi)$, $B(\sigma)$, $C$, $D$, $\prec$ in Table~\ref{tab:foata}, the modified Foata map $F_{\mathrm{mod}}$ is defined as follows. For $\pi=\pi_1\dots\pi_N\in \mathrm{Set}$, $F_{\mathrm{mod}}(\pi)=w^{(N)}$, where
$w^{(1)}=\pi_1$ and, for $2\le k\le N$, 
\[
w^{(k)}=
\begin{cases}
f(w^{(k-1)}\pi_k, \prec) & \mbox{if $\pi_k=C$ and $\pi_1\dots\pi_{k-1}$ does not have $D$,}\\
w^{(k)}=f(w^{(k-1)}\pi_k, <) & \mbox{otherwise.}
\end{cases}
\]
Then $F_{\mathrm{mod}}: \mathrm{Set}\to\mathrm{Set}$ is a bijection such that
if $\sigma=F_{\mathrm{mod}}(\pi)$, then $A(\pi)=B(\sigma)$. 
\end{thm}

We omit the proof of Theorem~\ref{thm:foata_mod} since it is similar to the proof of Theorem~\ref{thm:foata}.

\begin{table}
  \centering
\begin{tabular}{cclllll}
\hline
 $\mathrm{Set}$  &  $N$  &  $A(\pi)$  &  $B(\sigma)$  &  $C$  &  $D$  &  $\prec$  \\
\hline\
$\Alt^{-1}_{N}$ & $2n+1$ & $\maj(\pi)$ & $\inv(\sigma)$ & $-$ & $-$ & $-$ \\
 & & $\maj(\kappa_N\pi)$ & $\inv(\sigma)-\ndes((\sigma^{-1})_e)$ & $2i$ & $2i+2$ & $<_{2i}$\\
 & & $\maj(\eta_N\pi)$ & $\inv(\sigma)+\nasc((\sigma^{-1})_e)$ & $2i$ & $2i-2$ & $<_{2i-1}$ \\
 & $2n$ & $\maj(\pi)$ & $\inv(\sigma)$ & $-$ & $-$ & $-$  \\
 & & $\maj(\kappa_N\pi)$ & $\inv(\sigma)-\asc((\sigma^{-1})_*)$ & $2i$ & $2i+2$ & $<_{2i}$  \\
 & & $\maj(\eta_N\pi)$ & $\inv(\sigma)+\nasc((\sigma^{-1})_*)$ & $2i$ & $2i-2$ & $<_{2i-1}$ \\
\hline\
$\Ralt^{-1}_{N}$ & $2n+1$ & $\maj(\pi)$ & $\inv(\sigma)$ & $-$ & $-$ & $-$\\
 & & $\maj(\kappa_N\pi)$ & $\inv(\sigma)+\des((\sigma^{-1})_o)$ & $2i+1$ & $2i-1$ & $<_{2i}$ \\
 & & $\maj(\eta_N\pi)$ & $\inv(\sigma)-\asc((\sigma^{-1})_o)$ & $2i-1$ & $2i+1$ & $<_{2i-1}$ \\
 & $2n$ & $\maj(\pi)$ & $\inv(\sigma)$ & $-$ & $-$ & $-$  \\
 & & $\maj(\kappa_N\pi)$ & $\inv(\sigma)+\des((\sigma^{-1})_*)$ & $2i+1$ & $2i-1$ & $<_{2i}$ \\
 & & $\maj(\eta_N\pi)$ & $\inv(\sigma)-\ndes((\sigma^{-1})_*)$ & $2i-1$ & $2i+1$ & $<_{2i-1}$ \\
\hline\
\end{tabular}
  \caption{Each row defines the modified Foata map $F_{\mathrm{mod}}$ described in Theorem~\ref{thm:foata_mod}.
If $(C,D,\prec)=(-,-,-)$, $F_{\mathrm{mod}}$ is the original Foata map.}
  \label{tab:foata}
\end{table}


\section{Proof of Theorems~\ref{conj:9.3} and \ref{conj:9.6}}
\label{sec:MPP conjectures}

In this section we prove the two conjectures of Morales et al., Theorems~\ref{conj:9.3}~and~\ref{conj:9.6}. 
Let us briefly outline our proof. Recall that both Theorems~\ref{conj:9.3} and \ref{conj:9.6} are of the form $A=Q\det(c_{ij})$. 
In Section~\ref{sec:pleas-diagr-delt} we interpret pleasant diagrams of $\delta_{n+2k}/\delta_n$ as non-intersecting marked Dyck paths.
This interpretation can be used to express $A$ as a generating function for non-intersecting Dyck paths. 
In Section~\ref{sec:modif-lindstr-gess} we show a modification of Lindstr\"om--Gessel--Viennot lemma which allows us to express
$\det(c_{ij})$ as a generating function for weakly non-intersecting Dyck paths. 
In Section~\ref{sec:conn-betw-weakly} we find a connection between the generating function for weakly non-intersecting Dyck paths 
and the generating function for (strictly) non-intersecting Dyck paths. Using these results we prove
Theorems~\ref{conj:9.3}~and~\ref{conj:9.6} in Sections~\ref{sec:proof-th1} and \ref{sec:proof-thm2}.

Recall that $\Dyck_{2n}$ (resp.~$\Sch_{2n}$) is the set of Dyck (resp.~Schr\"oder) paths from $(-n,0)$ to $(n,0)$. Since a Dyck path is a Schr\"oder path without horizontal steps we have $\Dyck_{2n}\subseteq\Sch_{2n}$. 

For a Schr\"oder path $S\in\Sch_{2n}$ and $-n\le i\le n$, we define $S(i)=j$ if $(i,j)\in S$ or $\{(i-1,j),(i+1,j)\}$ is a horizontal step of $S$. For $S_1\in\Sch_{2n}$ and $S_2\in\Sch_{2n+4k}$,
we write $S_1\le S_2$ if $S_1(i)\le S_2(i)$ for all $-n\le i\le n$
and there is no $i$ such that $S_1(i)=S_2(i)$ and $S_1(i+1)= S_2(i+1)$.
Similarly, we write
$S_1< S_2$ if $S_1(i)< S_2(i)$ for all $-n\le i\le n$.
Note that if $S_1<S_2$, then $S_2$ is strictly above $S_1$.
If $S_1\le S_2$, then $S_2$ is weakly above $S_1$ and in addition $S_1$ and $S_2$ do not share any steps.

For a Dyck path $D\in\Dyck_{2n}$, we denote by $\VV(D)$ (resp.~$\HP(D)$) the set of valleys (resp.~high peaks) of $D$.
 
We denote by $\Dyck_{2n}^k$ the set of $k$-tuples $(D_1,\dots,D_k)$ of Dyck paths, where for $i\in[k]$,
\[
D_i\in\Dyck_{2n+4i-4} = \Dyck((-n-2i+2,0)\to(n+2i-2,0)).
\]
For brevity, if $(D_1,\dots,D_k)\in\Dyck_{2n}^k$
and $D_1\prec D_2\prec \dots\prec D_k$, where $\prec$ can be $<$ or $\le$, we will simply write
$(D_1\prec D_2\prec \dots\prec D_k)\in\Dyck_{2n}^k$. For example, 
$(D_1\le D_2<D_3)\in\Dyck_{2n}^3$ means $(D_1,D_2,D_3)\in\Dyck_{2n}^3$ and $D_1\le D_2<D_3$.

\subsection{Pleasant diagrams of $\delta_{n+2k}/\delta_n$ and non-intersecting marked Dyck paths}\
\label{sec:pleas-diagr-delt}

For a point $p=(i,j)\in\ZZ\times\NN$, the \emph{height} $\HT(p)$ of $p$ is defined to be $j$. We identify the square $u=(i,j)$ in the $i$th row and $j$th column in $\delta_{n+2k}$ with the point $p=(j-i,n+2k-i-j)\in\ZZ\times\NN$. Under this identification one can easily check that if a square $u\in \delta_{n+2k}$ corresponds to a point $p\in \ZZ\times\NN$ then the hook length $h(u)$ in $\delta_{n+2k}$ is equal to $2\HT(p)+1$.

We define the set $\ND_{2n}^{k}$ of non-intersecting Dyck paths by
\[
\ND_{2n}^{k} = \{(D_1,\dots,D_k)\in\Dyck_{2n}^k : D_1<\dots<D_k\}.
\]
Morales et al. showed that there is a natural bijection between $\ND_{2n}^k$ and $\EE(\delta_{n+2k}/\delta_n )$ as follows. 

\begin{prop}\cite[Corollary~8.4]{MPP2} \label{prop:excited}
The map $\rho:\ND_{2n}^k\to \EE(\delta_{n+2k}/\delta_n)$ defined by
\[
\rho(D_1,\dots,D_k) = \delta_{n+2k}\setminus (D_1\cup \dots \cup D_k)
\]
is a bijection. 
\end{prop}

A \emph{marked Dyck path} is a Dyck path in which each point that is not a valley may or may not be marked. 
Let
\[
\ND_{2n}^{*k} = \left\{(D_1,\dots,D_k,C): (D_1,\dots,D_k)\in\ND_{2n}^k, 
C\subset \bigcup_{i=1}^k (D_i \setminus \VV(D_i))\right\}.
\]
Note that an element $(D_1,\dots,D_k,C)\in \ND_{2n}^{*k}$ can be considered
as non-intersecting marked Dyck paths $D_1,\dots,D_k$ with the set $C$ of marked points. 
The following proposition allows us to consider pleasant diagrams of $\delta_{n+2k}/\delta_n$ as non-intersecting marked Dyck paths.

\begin{prop}\label{prop:pleasant}
The map $\rho^*:\ND_{2n}^{*k}\to \PP(\delta_{n+2k}/\delta_n)$ defined by
\[
\rho^*(D_1,\dots,D_k,C) = (D_1\cup \dots \cup D_k)\setminus C
\]
is a bijection. 
\end{prop}
\begin{proof}
By Proposition~\ref{prop:excited}, $\rho^*$ is a well defined map 
from $\ND_{2n}^{*k}$ to $\PP(\delta_{n+2k}/\delta_n)$.
We need to show that $\rho^*$ is surjective and injective. 

Consider an arbitrary pleasant diagram $P\in \PP(\delta_{n+2k}/\delta_n)$. 
By Proposition~\ref{prop:excited}, there is $(D_1,\dots,D_k)\in\ND_{2n}^k$ such that
$P\subset (D_1\cup \dots\cup D_k)$. We can take $(D_1,\dots,D_k)$ so that the squares in
$D_1,\dots, D_k$ are as northwest as possible. Then $P$ contains all valleys in $D_1,\dots,D_k$ because otherwise we make $D_1,\dots,D_k$ more toward northwest. 
Therefore $\rho^*$ is surjective. 

Now suppose that 
\[
\rho^*(D_1,\dots,D_k,C) = \rho^*(D_1',\dots,D_k',C').
\]
If $D_1\ne D_1'$, we can find a valley $p$ in one of $D_1$ and $D_1'$ that is not contained in the other Dyck path. We can assume $p\in \VV(D_1)$ and $p\not\in D'_1$. Then
we have $p\in \rho^*(D_1,\dots,D_k,C)$ and
$p\not\in\rho^*(D_1',\dots,D'_k,C')$, which is a contradiction. Thus we must have $D_1=D_1'$. Similarly we can show that $D_i=D_i'$ for all $i\in[k]$. Then we also have $C=C'$.
Thus $\rho^*$ is injective. 
\end{proof}

\begin{remark}
In Proposition~\ref{prop:pleasant}, we can also use high peaks instead of valleys.
This was shown for $k=1$ in \cite[Corollary~9.2]{MPP2}.
\end{remark}

\subsection{A modification of Lindstr\"om--Gessel--Viennot lemma}\
\label{sec:modif-lindstr-gess}

Let $\wt$ and $\wtext$ be fixed  weight functions defined on $\ZZ\times\NN$.
We define
\[
\wt_\VV(D) = \prod_{p\in D} \wt(p) \prod_{p\in\mathcal{V}(D)} \wtext(p)
\]
and
\[
\wt_\HP(D) = \prod_{p\in D} \wt(p) \prod_{p\in\HP(D)} \wtext(p).
\]
One can regard $\wt_\VV(D)$ as a weight of a Dyck path $D$ in which every point $p$ of $D$ has the weight $\wt(p)$
and every valley $p$ of $D$ has the extra weight $\wtext(p)$.  For Dyck paths $D_1, \dots, D_k$, we define
\[
\wt_\VV(D_1, \dots, D_k) = \wt_\VV(D_1)\cdots \wt_\VV(D_k).
\]

The next lemma is a modification of Lindstr\"om--Gessel--Viennot lemma. 

\begin{lem}\label{lem:det}
For $1\le i,j\le k$, let $A_i=(-n-2i+2,0)$, $B_j=(n+2j-2,0)$ and
\[
d_n^{i,j}(q) = \sum_{D\in\Dyck(A_i\to B_j)} \wt_\VV(D).
\]
Then 
\begin{equation}\label{eqn:lem_det}
\det(d_{n}^{i,j}(q))_{i,j=1}^k  
=\sum_{(D_1\le \dots \le D_k)\in\Dyck_{2n}^k} \wt_\VV(D_1, \dots, D_k) \prod_{i=1}^{k-1} \prod_{p\in D_i\cap D_{i+1}} \left(1-\frac{1}{\wtext(p)}\right).
\end{equation}
\end{lem}

Note that if $\wt$ and $\wtext$ depend only on the $y$-coordinates, i.e.,
$\wt(a,b)=f(b)$ and $\wtext(a,b)=g(b)$ for some functions $f$ and $g$, then 
$d_{n}^{i,j}(q)$ can be written as $d_{n+i+j-2}(q)$, where
\[
d_n(q) = \sum_{D\in\Dyck_{2n}} \wt_\VV(D).
\]

\begin{proof}
We have
\[
\det(d_{n}^{i,j}(q))_{i,j=1}^k =\sum_{(D_1,\dots,D_k)\in T} w(D_1,\dots,D_k),
\]
where $T$ is the set of $k$-tuples $(D_1,\dots,D_k)$ of Dyck paths such that
$D_i\in\Dyck(A_i\to B_{\pi_i})$ for some permutation $\pi\in \Sym_k$ and
$w(D_1,\dots,D_k) = \sgn(\pi)  \wt_\VV(D_1, \dots, D_k) $.
Let $T_1$ be the subset of $T$ consisting of $(D_1,\dots,D_k)$ such that
$D_i$ and $D_j$ do not share any common step for all $1\le i<j\le k$.

Suppose that $(D_1,\dots,D_k)\in T\setminus T_1$.  Then we can find the lexicographically largest index $(i,j)$ such that $D_i$ and $D_j$ have a common step. Let $p_1$ and $p_2$ be the rightmost two consecutive points contained in $D_i$ and $D_j$. 
Let $D_i'$ and $D_j'$ be the Dyck paths obtained from $D_i$ and $D_j$, respectively, by exchanging the points after $p_2$.
For $t\ne i,j$, let $D_t'=D_t$. Then the map $(D_1,\dots,D_k)\mapsto (D_1',\dots,D_k')$ is a sign-reversing involution on $T\setminus T_1$ with no fixed points. Therefore, we have
\[
\sum_{(D_1,\dots,D_k)\in T} w(D_1,\dots,D_k) = \sum_{(D_1,\dots,D_k)\in T_1} w(D_1,\dots,D_k).
\]

Now consider $(D_1,\dots,D_k)\in T_1$. By the condition on the elements in $T_1$, we can find a unique $k$-tuple
$(P_1\le \dots\le P_k)\in \Dyck_{2n}^k$ such that the steps in $D_1,\dots,D_k$ are the same as the steps in $P_1,\dots,P_k$.
Note that we also have $(P_1, \dots, P_k)\in T_1$. 
We define the labeling $L$ of the intersection points in $P_1,\dots,P_k$ as follows. Let $p$ be an intersection point in $P_1,\dots,P_k$. 
Then $p$ is also an intersection point of two paths in $(D_1,\dots,D_k)$. Let $D_r$ and $D_s$ be the Dyck paths containing $p$. 
Since $P_1\le \dots\le P_k$ and $(P_1, \dots, P_k)\in T_1$, there is a unique $i$ for which $p$ is contained in $P_i$ and $P_{i+1}$.
Then $p$ is both a peak of $P_{i}$ and a valley of $P_{i+1}$. On the other hand, $p$ may or may not be a peak or a valley in $D_r$ or $D_s$. 
Observe that $p$ is a peak (resp.~valley) of $D_r$ if and only if it is a valley (resp.~peak) of $D_s$. 
We define the label $L(p)$ of $p$ by 
\[
L(p) =
\begin{cases}
  1 & \mbox{if $p$ is a valley of $D_r$ or $D_s$,}\\
-1/\wtext(p) & \mbox{otherwise.}
\end{cases}
\]
For example, if $D_1$ and $D_2$ are the paths in Figure~\ref{fig:D1D2}, the paths $P_1$ and $P_2$ and the labeling $L$ are determined as shown  in Figure~\ref{fig:P1P2}.
\begin{figure}
  \centering
\includegraphics{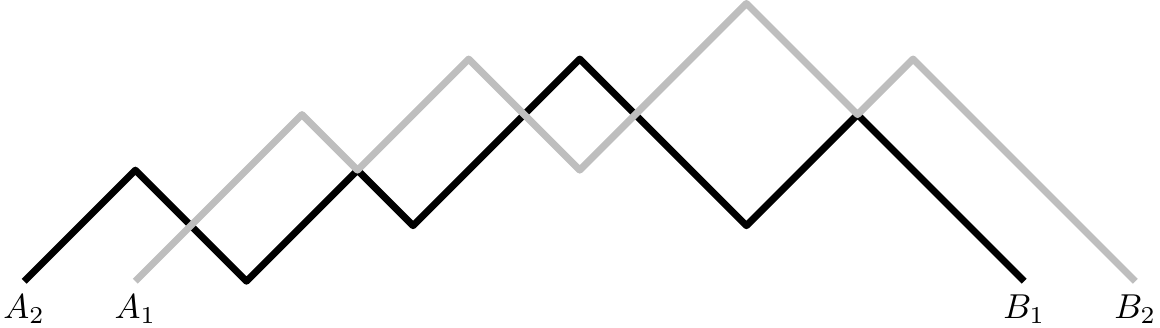}  
  \caption{An example of $(D_1,D_2)\in T_1$ in the proof of Lemma~\ref{lem:det}. The gray path is $D_1$ and the black path is $D_2$.}
  \label{fig:D1D2}
\end{figure}
\begin{figure}
  \centering
\includegraphics{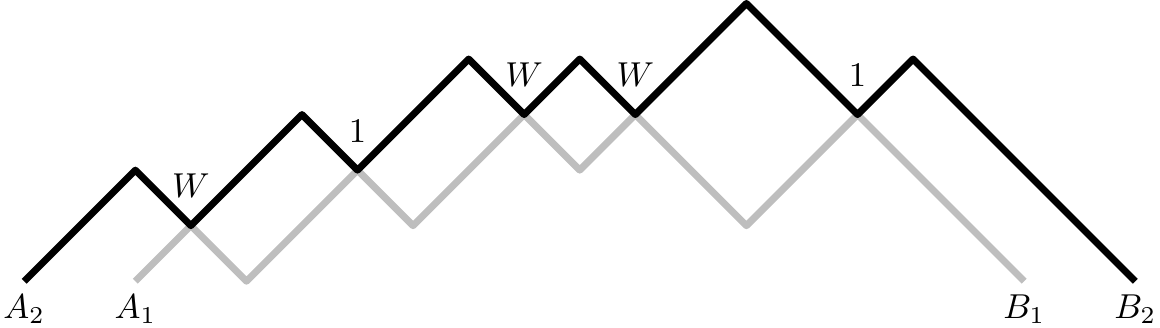}  
  \caption{The gray path is $P_1$ and the black path is $P_2$. The letter $W$ above an intersection point $p$ means that $L(p)=-1/\wtext(p)$.}
  \label{fig:P1P2}
\end{figure}

Let $w'(P_1,\dots,P_k,L)$ be the product of 
$\wt_\VV(D_1, \dots, D_k)$ and the labels $L(p)$ of all intersection points $p$. It is not hard to see that
$w(D_1,\dots,D_k)=w'(P_1,\dots,P_k,L)$. Moreover, given $(P_1\le\cdots\le P_k)\in \Dyck_{2n}^k$ and
any labeling $L$ of their intersection points, there is a unique $(D_1,\dots,D_k)\in T_1$ which makes
$P_1,\dots,P_k$ and the labeling $L$. Thus, 
\[
\sum_{(D_1,\dots,D_k)\in T_1} w(D_1,\dots,D_k) 
=\sum_{P_1,\dots,P_k,L} w'(P_1,\dots,P_k,L),
\]
where the sum is over all $(P_1\le\cdots\le P_k)\in \Dyck_{2n}^k$ and all possible labelings $L$ of their intersection points.
Since
\[
\sum_{P_1,\dots,P_k,L} w'(P_1,\dots,P_k,L) = \sum_{(D_1\le \dots \le D_k)\in\Dyck_{2n}^k} \wt_\VV(D_1, \dots, D_k) \prod_{i=1}^{k-1} \prod_{p\in D_i\cap D_{i+1}} \left(1-\frac{1}{\wtext(p)}\right),
\]
we obtain the desired identity.
\end{proof}

\begin{remark}\label{rmk:LGV}
Lindstr\"om--Gessel--Viennot lemma \cite{Lindstrom,GesselViennot}
expresses a determinant as a sum over non-intersecting lattice paths. In our case, due to the extra weights on the valleys, the paths which have common points are not completely cancelled. Therefore the right-hand side of \eqref{eqn:lem_det} is a sum over \emph{weakly} non-intersecting lattice paths.
If $\wtext(p)=1$ for all points $p$, then Lemma~\ref{lem:det} reduces to Lindstr\"om--Gessel--Viennot lemma.
\end{remark}

\subsection{Weakly and strictly non-intersecting Dyck paths}\
\label{sec:conn-betw-weakly}

For a constant $c$ and a Schr\"oder path $S$, we define 
\[
\wt_\Sch(c,S) = c^{\hs(S)} \prod_{p\in S} \wt(p),
\]
where $\hs(S)$ is the number of horizontal steps in $S$. In other words, we give extra weight $c$ to each horizontal step.

For a Schr\"oder path $S$, we define $\phi_\VV(S)$ (resp. $\phi_\HP(S)$) to be the Dyck path obtained from $S$ by changing each horizontal step 
to a down step followed by an up step (an up step followed by a down step).
For a Dyck path $D$, we denote by $\phi_\VV^{-1}(D)$ (resp.~$\phi_\HP^{-1}(D)$) the inverse image of $D$, i.e., the set of all Schr\"oder paths $S$ with
$\phi_\VV(S)=D$ (resp.~$\phi_\HP(S)=D$). 

Note that for a given Dyck path $D$, there can be several Schr\"oder paths $S$ mapping to $D$ via $\phi_\VV$ or $\phi_\HP$. The following proposition tells us that if $\wt(p)\left( \wtext(p)-1 \right)$ is constant, then the weight $\wt_\VV(D)$ or $\wt_\HP(D)$ is equal 
to the sum of weights of such Schr\"oder paths. 

\begin{prop}\label{prop:MPP_weight}
Suppose that the weight functions $\wt$ and $\wtext$ satisfy
$\wt(p)\left( \wtext(p)-1 \right)=c$ for all $p\in\ZZ\times\NN$. Then, for every Dyck path $D$, we have
\[
\wt_\VV(D) = \sum_{S\in \phi_\VV^{-1}(D)} \wt_\Sch(c,S) \qand \wt_\HP(D) = \sum_{S\in \phi_\HP^{-1}(D)} \wt_\Sch(c,S).
\]
\end{prop}
\begin{proof}
The elements in $\phi_\VV^{-1}(D)$ (resp.~$\phi_\HP^{-1}(D)$) are those obtained from $D$
by doing the following: for each point $p$ in $\VV(D)$ (resp.~$\HP(D)$), either do nothing or remove $p$. 
Here, if $p$ is in $\VV(D)$ (resp.~$\HP(D)$), removing $p$ can be understood as replacing the pair of a down step and an up step (resp.~an up step and a down step) at $p$ by a horizontal step. Thus
\begin{align*}
\sum_{S\in \phi_\VV^{-1}(D)} \wt_\Sch(c,S) &= \prod_{v\in\mathcal{V}(D)} (\wt(v) + c) \prod_{p\in D\setminus\mathcal{V}(D)} \wt(p) \\
	&= \prod_{v\in\mathcal{V}(D)} \left( \frac{\wt(v) + c}{\wt(v)} \right) \prod_{p\in D} \wt(p) \\
	&= \prod_{v\in\mathcal{V}(D)} \left(1 + \frac{\wt(v)(\wtext(v)-1)}{\wt(v)} \right) \prod_{p\in D} \wt(p) \\
	&= \prod_{v\in\mathcal{V}(D)} \wtext(v) \prod_{p\in D} \wt(p)\\
	&= \wt_\VV(D).
\end{align*}

The second identity can be proved in the same way using high peaks instead of valleys.
\end{proof}

The following proposition is the key ingredient for the proofs of Theorems~\ref{conj:9.3} and \ref{conj:9.6}.

\begin{prop} \label{prop:MPP_key1}
Suppose that the weight functions $\wt$ and $\wtext$ satisfy
$\wt(p)\left( \wtext(p)-1 \right)=c$ for all $p\in\ZZ\times\NN$. 
Let $A\in\Dyck_{2n}$ and $B\in\Dyck_{2n+8}$ be fixed Dyck paths with $A<B$. 
Then
\begin{multline*}
\sum_{(A\le D<B)\in\Dyck_{2n}^3} \wt_\VV(D) \prod_{p\in A\cap D} \left( 1 - \frac{1}{\wtext(p)} \right) \\
= \sum_{(A<D\le B)\in\Dyck_{2n}^3} \wt_\HP(D) \prod_{p\in D\cap B} \left( 1 - \frac{1}{\wtext(p)} \right).
\end{multline*}
\end{prop}
\begin{proof}
Let 
\begin{align*}
X&=\{D\in\Dyck_{2n+4}:A\le D<B\},  \\
Y&=\{D\in\Dyck_{2n+4}:A< D\le B\},\\
Z&=\{S\in\Sch_{2n+4}:A< S< B\}.
\end{align*}
It suffices to show the following identities
\begin{equation}
  \label{eq:XZ}
\sum_{S\in Z}  \wt_\Sch(c,S)
=\sum_{D\in X}  \wt_\VV(D) \prod_{p\in A\cap D} \left( 1 - \frac{1}{\wtext(p)} \right)
\end{equation}
and
\begin{equation}
  \label{eq:YZ}
\sum_{S\in Z}  \wt_\Sch(c,S)
=\sum_{D\in Y}  \wt_\HP(D) \prod_{p\in D\cap B} \left( 1 - \frac{1}{\wtext(p)} \right).
\end{equation}

Note that if $D\in Y$, every peak of $D$ is a high peak. 
It is easy to see that for all $S\in Z$, we have $\phi_\VV(S)\in X$ and $\phi_\HP(S)\in Y$.
Consider the restrictions $\phi_\VV |_Z:Z\to X$ and $\phi_\HP |_Z:Z\to Y$, which are surjective.
Observe that for $D\in X$ (resp.~$D\in Y$), we have $A\cap D\subset \VV(D)$ (resp.~$D\cap B\subset \HP(D)$).
The elements in $(\phi_\VV|_Z)^{-1}(D)$ (resp.~$(\phi_\HP|_Z)^{-1}(D)$) are those obtained from $D$
by doing the following: for each point $p$ in $\VV(D)\setminus A$ (resp.~$\HP(D)\setminus B$), either do nothing or remove $p$,
and for each point $p$ in $A\cap D$ (resp.~$D\cap B$), remove $p$.

Now we prove \eqref{eq:XZ}. Let $D\in X$. 
By the same arguments used in the proof of Proposition~\ref{prop:MPP_weight}, we have
\begin{align*}
\sum_{S\in (\phi_\VV|_Z)^{-1}(D)} \wt_\Sch(c,S) &=  \prod_{p\in D\setminus\mathcal{V}(D)} \wt(p)  \prod_{v\in\mathcal{V}(D)\setminus A} (\wt(v) + c) 
\prod_{p\in A\cap D} c\\
&=\prod_{p\in D\setminus\mathcal{V}(D)} \wt(p)  \prod_{v\in\mathcal{V}(D)} \wt(v)\wtext(v)
\prod_{p\in A\cap D} \frac{\wt(p)\left( \wtext(p)-1 \right)}{\wt(p)\wtext(p)}\\
	&= \wt_\VV(D) \prod_{p\in A\cap D} \left( 1 - \frac{1}{\wtext(p)} \right).
\end{align*}
Thus, we obtain
\[
\sum_{S\in Z}  \wt_\Sch(c,S)
=\sum_{D\in X} \sum_{S\in (\phi_\VV|_Z)^{-1}(D)} \wt_\Sch(c,S) 
=\sum_{D\in X}  \wt_\VV(D) \prod_{p\in A\cap D} \left( 1 - \frac{1}{\wtext(p)} \right),
\]
which is \eqref{eq:XZ}. 
The second identity~\eqref{eq:YZ} can be proved similarly. 
\end{proof}

If $\wt$ and $\wtext$ satisfy certain conditions, we have a connection
between weakly non-intersecting Dyck paths and strictly non-intersecting Dyck paths as follows. 

\begin{prop} \label{prop:MPP_key2}
Suppose that $\wt$ and $\wtext$ satisfy the following conditions
\begin{itemize}
\item $\wt(p)\left( \wtext(p)-1 \right)=c$ for all $p\in\ZZ\times\NN$, and 
\item $\wt_\HP(D) = t_j \wt_\VV(D)$ for all $D\in\Dyck_{2j}$ such that every peak in $D$ is a high peak.
\end{itemize}
Then we have
\begin{multline*}
\sum_{(D_1\le \cdots \le D_k)\in\Dyck_{2n}^k} \wt_\VV(D_1, \dots, D_k) \prod_{i=1}^{k-1} \prod_{p\in D_i\cap D_{i+1}} \left( 1 - \frac{1}{\wtext(p)} \right) \\
= \prod_{i=1}^{k-1} t_{n+2i}^{i} \sum_{(D_1 < \cdots < D_k)\in\Dyck_{2n}^k} \wt_\VV(D_1, \dots, D_k).
\end{multline*}
\end{prop}
\begin{proof}
Throughout this proof, we assume that $D_i\in \Dyck_{2n+4i-4}$ for $i\in[k]$
and $D_{k+1}$ is fixed to be the highest Dyck path in $\Dyck_{2n+4k}$ consisting of $n+2k$ up steps followed by $n+2k$ down steps. For a sequence $(\prec_1,\dots,\prec_k)$ of inequalities $<$ and $\le$, let
\[
g(\prec_1,\dots,\prec_k) = \sum_{D_1\prec_1 D_2 \prec_2\cdots \prec_k D_{k+1}} \wt_\VV(D_1, \dots, D_k) \prod_{i=1}^{k-1} \prod_{p\in D_i\cap D_{i+1}} \left( 1 - \frac{1}{\wtext(p)} \right).
\]
Note that we always have
\[
g(\prec_1,\dots,\prec_{k-1},<) = g(\prec_1,\dots,\prec_{k-1},\le).
\]
The identity in this proposition can be restated as
\begin{equation}
  \label{eq:8}
g(\le,\dots,\le)  = \prod_{i=1}^{k-1} t_{n+2i}^{i} \cdot g(<,\dots,<) .
\end{equation}

Suppose that $D_{i-1}$ and $D_{i+1}$ are fixed, where $2\le i\le k$. Then, by Proposition~\ref{prop:MPP_key1} and the assumptions on $\wt$ and $\wtext$, we have
\begin{multline*}
\sum_{D_{i-1}\le D_i<D_{i+1}} \wt_\VV(D_i) \prod_{p\in D_{i-1}\cap D_i} \left( 1 - \frac{1}{\wtext(p)} \right)\\
= t_{n-2+2i} \sum_{D_{i-1}< D_i\le D_{i+1}} \wt_\VV(D_i) \prod_{p\in D_{i}\cap D_{i+1}} \left( 1 - \frac{1}{\wtext(p)} \right).
\end{multline*}
This implies that for any $2\le i\le k$ and
$\prec_1,\dots,\prec_{i-2},\prec_{i+1},\dots,\prec_k$, 
\[
g(\prec_1,\dots,\prec_{i-2},\le,<,\prec_{i+1},\dots,\prec_k) = 
t_{n+2i-2}\cdot g(\prec_1,\dots,\prec_{i-2},<,\le,\prec_{i+1},\dots,\prec_k).
\]
Therefore,
\begin{align*}
g(\le,\dots,\le)  &= g(\le,\dots,\le,<)
= t_{n+2k-2}\cdot g(\le,\dots,\le,<,\le)\\
&= t_{n+2k-2}t_{n+2k-4}\cdot g(\le,\dots,\le,<,\le,\le)\\
&= \cdots \\
& = \prod_{i=1}^{k-1} t_{n+2i} \cdot g(<,\le,\dots,\le).
\end{align*}
By applying this process repeatedly we obtain \eqref{eq:8}. 
\end{proof}

\subsection{Proof of Theorem~\ref{conj:9.3}}
\label{sec:proof-th1}

Recall the first conjecture of Morales, Pak and Panova (Theorem~\ref{conj:9.3}) :
\begin{equation}
  \label{eq:6}
p(\delta_{n+2k}/\delta_n) = 2^{\binom k2} \det \left( \mathfrak{s}_{n-2+i+j} \right)_{i,j=1}^k,
\end{equation}
where $\mathfrak{s}_{n} = p(\delta_{n+2}/\delta_n)$.
By Proposition~\ref{prop:pleasant},
\begin{align*}
p(\delta_{n+2k}/\delta_n) &= \sum_{(D_1< \dots < D_k)\in\Dyck_{2n}^k} 
\prod_{i=1}^k 2^{|D_1\setminus \VV(D_1)|}\\
&= 2^{k(2n-3)+4\binom{k+1}2}  
\sum_{(D_1< \dots < D_k)\in\Dyck_{2n}^k} 
\left(\frac{1}{2}\right)^{v(D_1)+\dots+v(D_k)}
\end{align*}
and
\[
\mathfrak{s}_{n} = p(\delta_{n+2}/\delta_n) = 2^{2n+1}\sum_{D\in\Dyck_{2n}} \left(\frac 12\right)^{v(D)}.
\]
Let 
\[
d_n(q) = \sum_{D\in\Dyck_{2n}}q^{v(D)}
\]
and 
\[
d_{n,k}(q)=\sum_{(D_1<D_2<\dots<D_k)\in \Dyck_{2n}^k}q^{v(D_1)+\dots+v(D_k)}.
\]
Then \eqref{eq:6} can be rewritten as 
\[
2^{k(2n-3)+4\binom{k+1}2} d_{n,k}(1/2)=
2^{\binom k2} \det(2^{2n-3+2i+2j}d_{n+i+j-2}(1/2))_{i,j=1}^k
\]
or
\[
2^{-\binom k2} d_{n,k}(1/2)= \det(d_{n+i+j-2}(1/2))_{i,j=1}^k.
\]
Thus Theorem~\ref{conj:9.3} is obtained from the following theorem by substituting $q=1/2$.
\begin{thm}\label{thm:MPP1_restated}
For $n,k\ge1$, we have
\[
\det(d_{n+i+j-2}(q))_{i,j=1}^k = q^{\binom k2} d_{n,k}(q).
\]
\end{thm}

\begin{proof}
By Lemma~\ref{lem:det} with $\wt(p) = 1$ and $\wtext(p) = q$ for all $p\in\ZZ\times\NN$, we can rewrite Theorem~\ref{thm:MPP1_restated} as 

\begin{multline}\label{eqn:MPP1}
\sum_{(D_1\le \dots \le D_k)\in\Dyck_{2n}^k} q^{v(D_1)+\dots+v(D_k)}
\left(1-\frac{1}{q}\right)^{|D_1\cap D_2|+\dots+|D_{k-1}\cap D_k|} \\
=q^{\binom k2} \sum_{(D_1< \dots < D_k)\in\Dyck_{2n}^k} q^{v(D_1)+\dots+v(D_k)}.
\end{multline}
Then \eqref{eqn:MPP1} follows immediately from Proposition~\ref{prop:MPP_key2} and  Lemma~\ref{lem:10} below.
\end{proof}

\begin{lem} \label{lem:10}
Let $\wt(p) = 1$ and $\wtext(p) = q$ for a lattice point $p$. Then 
\begin{itemize}
\item $\wt(p)\left( \wtext(p)-1 \right)=q-1$ for all lattice points $p$, and
\item $\wt_\HP(D) = q \wt_\VV(D)$ for all $D\in\Dyck_{2j}$ such that every peak in $D$ is a high peak.
\end{itemize}
\end{lem}
\begin{proof}
The first statement is clear. The second one follows from the fact that 
in a Dyck path the number of peaks is always one more than the number of valleys. 
\end{proof}

\subsection{Proof of Theorem~\ref{conj:9.6}}
\label{sec:proof-thm2}

Recall the second conjecture of Morales, Pak and Panova (Theorem~\ref{conj:9.6}) :
\begin{equation}
  \label{eq:7}
\sum_{\pi\in \RPP(\delta_{n+2k}/\delta_n)} q^{|\pi|} 
=q^{-\frac{k(k-1)(6n+8k-1)}{6}} \det\left( \frac{E^*_{2n+2i+2j-3}(q)}{(q;q)_{2n+2i+2j-3}} \right)_{i,j=1}^k.
\end{equation}

By \eqref{eq:MPP2} and Proposition~\ref{prop:pleasant}, we have
\begin{align*}
&\sum_{\pi\in \RPP(\delta_{n+2k}/\delta_n)} q^{|\pi|}  
=\sum_{P\in\PP(\delta_{n+2k}/\delta_n)} \prod_{u\in P} \frac{q^{h(u)}}{1-q^{h(u)}}\\
&= \sum_{(D_1< \dots < D_k)\in\Dyck_{2n}^k}  \prod_{i=1}^k \left(
\prod_{p\in \VV(D_i)} \frac{ q^{2\HT(p)+1}}{1-q^{2\HT(p)+1}}
\prod_{p\in D_i\setminus\VV(D_i)} \left( 1+ \frac{ q^{2\HT(p)+1}}{ 1-q^{2\HT(p)+1}}\right)
\right)\\
&= \sum_{(D_1< \dots < D_k)\in\Dyck_{2n}^k}  \prod_{i=1}^k \left(
\prod_{p\in \VV(D_i)} q^{2\HT(p)+1} \prod_{p\in D_i} \frac{1}{1-q^{2\HT(p)+1}}
\right)
\end{align*}
and
\[
\frac{E^*_{2n+1}(q)}{(q;q)_{2n+1}} = \sum_{\pi\in \RPP(\delta_{n+2}/\delta_n)} q^{|\pi|} = 
\sum_{D\in\Dyck_{2n}} \prod_{p\in \VV(D)} q^{2\HT(p)+1} \prod_{p\in D} \frac{1}{1-q^{2\HT(p)+1}}.
\]
Thus, by Lemma~\ref{lem:det} with $\wt(p) = 1/(1-q^{2\HT(p)+1})$ and $\wtext(p) = q^{2\HT(p)+1}$, we can rewrite \eqref{eq:7} as follows.

\begin{thm}\label{thm:MPP2}
We have
\begin{multline*}
\sum_{(D_1\le \dots \le D_k)\in\Dyck_{2n}^k} \prod_{i=1}^k \left(
\prod_{p\in \VV(D_i)} q^{2\HT(p)+1} \prod_{p\in D_i} \frac{1}{1-q^{2\HT(p)+1}} \right)\prod_{j=1}^{k-1} \prod_{p\in D_j\cap D_{j+1}} \left(1 - \frac{1}{q^{2\HT(p)+1}}\right) \\
=q^{\frac{k(k-1)(6n+8k-1)}{6}} \sum_{(D_1< \dots < D_k)\in\Dyck_{2n}^k}  \prod_{i=1}^k \left(
\prod_{p\in \VV(D_i)} q^{2\HT(p)+1} \prod_{p\in D_i} \frac{1}{1-q^{2\HT(p)+1}} \right).
\end{multline*}
\end{thm}

Theorem~\ref{thm:MPP2} follows immediately from Proposition~\ref{prop:MPP_key2} with $t_j=q^{2j+1}$, Lemma~\ref{lem:last} below and the fact
\[
\sum_{i=1}^{k-1} i(2n+4i+1) = \frac{k(k-1)(6n+8k-1)}{6}.
\]
\begin{lem} \label{lem:last}
Let $\wt(p) = 1/(1-q^{2\HT(p)+1})$ and $\wtext(p) = q^{2\HT(p)+1}$. Then 
\begin{itemize}
\item $\wt(p)\left( \wtext(p)-1 \right)=-1$ for all $p$, and 
\item $\wt_\HP(D) = q^{2n+1} \wt_\VV(D)$ for all $D\in\Dyck_{2n}$ such that every peak in $D$ is a high peak.
\end{itemize}
\end{lem}

\begin{proof}
The first statement is clear. For the second one, consider a Dyck path $D\in\Dyck_{2n}$ all of whose peaks are high peaks. 
Let $p_1=(x_1,y_1),\dots,p_k=(x_k,y_k)$ are the peaks of $D$ with $x_1<\dots<x_k$. Then the valleys of $D$ are
the points $v_i=(x'_i,y'_i)$ for $i\in[k-1]$ satisfying $y_i'-y_i = -(x_i'-x_i)$ and $y_i'-y_{i+1} = x_i'-x_{i+1}$. Solving the system of equations
for $x_i '$ and $y_i '$ gives 
\[
x'_i = \frac{x_i+x_{i+1}+y_i-y_{i+1}}2  \qand y'_i = \frac{x_i-x_{i+1}+y_i+y_{i+1}}2.
\]
Thus 
\[
\wt_\VV(D)=q^{\sum_{p\in\VV(D)}(2\HT(p)+1)} \wt(D) = q^{\sum_{i=1}^{k-1}(2y_i'+1)} \wt(D)
=q^{(x_1-x_k-y_1-y_k) + 2(y_1+\dots+y_k)+k-1} \wt(D).
\]
Since $(x_1,y_1)$ is the first peak and $(x_k,y_k)$ is the last peak, $x_1=y_1$ and $x_k+y_k=2n$. 
Therefore,
\[
\wt_\VV(D)=q^{-2n-1} q^{2(y_1+\dots+y_k)+k} \wt(D) = q^{-2n-1}\wt_\HP(D) . \qedhere
\]
\end{proof}

\section{A determinantal formula for a certain class of skew shapes}
\label{sec:lascoux-pragacz-type}

Lascoux and Pragacz \cite{Lascoux1988} found a determinantal formula for a skew Schur function in terms of ribbon Schur functions. 
Morales et al. \cite{MPP2} found a similar determinantal formula for the number of excited diagrams.
In this section, applying the same methods used in the previous section, we find a determinantal formula for $p(\lm)$ and the generating function for the reverse plane partitions of shape $\lm$ for a certain class of skew shapes $\lm$ including $\delta_{n+2k}/\delta_{n}$ and $\delta_{n+2k+1}/\delta_{n}$.

Consider a partition $\lambda$. Recall that $u=(i,j)\in\lambda$ means that $u$ is a cell in the $i$th row and $j$th column of the Young diagram of $\lambda$.  Let $L=(u_0,u_1,\dots,u_m)$ be a sequence of cells in $\lambda$.
Each pair $(u_{i-1},u_i)$ is called a \emph{step} of $L$. A step $(u_{i-1},u_i)$ is called an \emph{up step} (resp.~\emph{down step}) if $u_i-u_{i-1}$ is equal to $(-1,0)$ (resp.~$(0,1)$).
We say that $L$ is a \emph{$\lambda$-Dyck path} if every step is either
an up step or a down step. The set of $\lambda$-Dyck paths starting at a cell $s$ and ending at a cell $t$ is denoted by
$\Dyck_{\lambda}(s,t)$. If $s$ is the southmost cell of a column of $\lambda$
and $t$ is the eastmost cell of a row of $\lambda$, we denote by
$L_\lambda(s,t)$ the lowest Dyck path in $\Dyck_{\lambda}(s,t)$, see Figure~\ref{fig:lowest} for an example.

\begin{figure}[h]
  \centering
\includegraphics{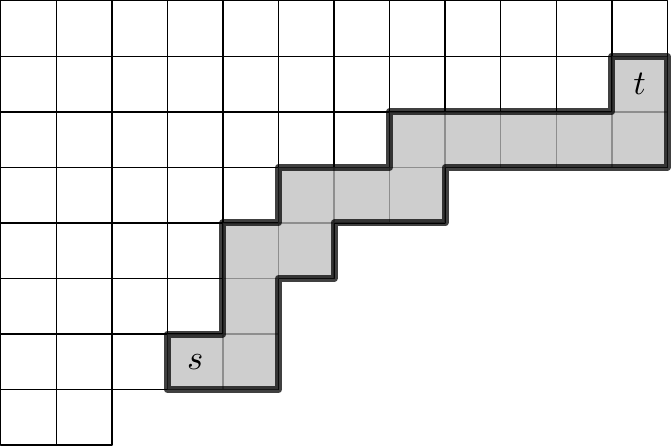}
  \caption{The lowest path $L_\lambda(s,t)$ in $\Dyck_{\lambda}(s,t)$. }
  \label{fig:lowest}
\end{figure}

Let $D=(u_0,u_1,\dots,u_m)$ be a $\lambda$-Dyck path. 
A cell $u_i$, for $1\le i\le m-1$, is called a \emph{peak} (resp.~\emph{valley}) if $(u_{i-1},u_i)$ is an up step (resp.~down step)
and $(u_{i},u_{i+1})$ is a down step (resp.~up step).
A peak $u_i$ is called a \emph{$\lambda$-high peak} if $u_i+(1,1)\in\lambda$. 
The set of valleys in $D$ is denoted by $\VV(D)$.

Suppose that $D_1$ and $D_2$ are $\lambda$-Dyck paths.
We say that $D_1$ is \emph{weakly below} $D_2$, denoted $D_1\le D_2$, if the following conditions hold.
\begin{itemize}
\item For every cell $u\in D_1$, there is an integer $k\ge0$ satisfying $u-(k,k)\in D_2$. 
\item If $u\in D_1\cap D_2$, then $u$ is a peak of $D_1$ and a valley of $D_2$. 
\end{itemize}
We say that $D_1$ is \emph{strictly below} $D_2$, denoted $D_1<D_2$, if $D_1\le D_2$ and $D_1\cap D_2=\emptyset$.

The \emph{Kreiman outer decomposition} of $\lm$ is a sequence $L_1,\dots,L_k$ of mutually disjoint nonempty $\lambda$-Dyck paths 
 satisfying the following conditions.
\begin{itemize}
\item Each $L_i$ starts at the southmost cell of a column of $\lambda$ and ends at the eastmost cell of a row of $\lambda$.
\item $L_1\cup \dots\cup L_k = \lm$.
\end{itemize}
See Figure~\ref{fig:Kreiman} for an example of the Kreiman outer decomposition.
It is known \cite{Kreiman} that there is a unique (up to permutation) Kreiman outer decomposition of $\lm$. 

\begin{figure}[h]
  \centering
\includegraphics{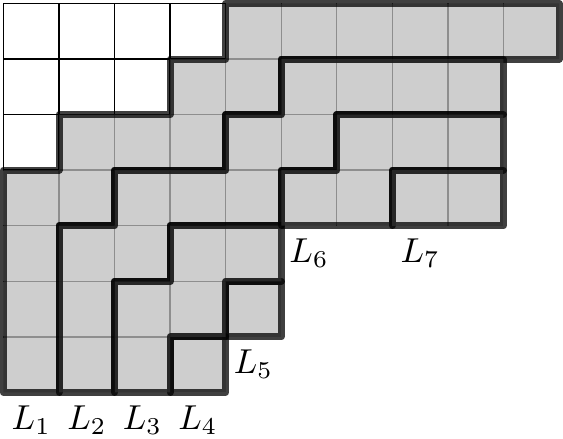} \qquad \qquad \includegraphics{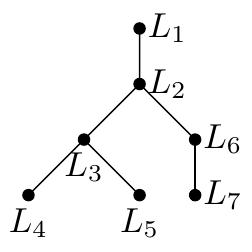}
  \caption{The left diagram shows the Kreiman outer decomposition $L_1,\dots,L_7$ of $\lm$ for $\lambda=(9,8,8,8,5,5,4)$ and $\mu=(4,3,1)$. The label $L_i$ is written below the starting cell of it. The right diagram shows the poset of $L_1,\dots,L_7$ with relation $<$. }
  \label{fig:Kreiman}
\end{figure}

A poset $P$ is called \emph{ranked} if there exists a rank function $r:P\to\NN$ satisfying the following conditions. 
\begin{itemize}
\item If $x$ is a minimal element, $r(x)=0$.
\item If $y$ covers $x$, then $r(y)=r(x)+1$. 
\end{itemize}
Note that if $P$ is a ranked poset, there is a unique rank function $r:P\to\NN$. The \emph{rank}
of $x\in P$ is defined to be $r(x)$. For example,
the poset in Figure~\ref{fig:Kreiman} is a ranked poset and the ranks of $L_1,L_2,\dots,L_7$ are $3,2,1,0,0,1,0$ respectively. However, 
if we remove the element $L_7$ from this poset, the resulting poset is not a ranked poset.

The following two theorems can be proved by the same arguments in the previous section. We omit the details.

\begin{thm}\label{thm:LP1}
Let $L_1,\dots,L_k$ be the Kreiman outer decomposition of $\lm$. 
Let $P$ be the poset of $L_1,\dots,L_k$ with relation $<$. Suppose that the following conditions hold.
\begin{itemize}
\item $P$ is a ranked poset.
\item If $L_i<L_j$, then in $L_j$ the first step is an up step, the last step is a down step and every peak is a $\lambda$-high peak.
\end{itemize}
Let $s_i$ (resp.~$t_i$) be the first (resp.~last) cell in $L_i$ and $r_i$  the rank of $L_i$ in the poset $P$. 
Then we have
\[
\sum_{\pi\in\RPP(\lm)} q^{|\pi|} = q^{-\sum_{i=1}^k r_i|L_i|} \det \left(
E_\lambda(s_i,t_j;q)\right)_{i,j=1}^k,
\]
where 
\[
E_\lambda(s_i,t_j;q) = \sum_{\pi\in \RPP(L_\lambda(s_i,t_j))} q^{|\pi|} = 
\sum_{D\in\Dyck_\lambda(s_i,t_j)} \prod_{u\in D} \frac1{1-q^{h(u)}} \prod_{u\in\VV(D)} q^{h(u)}.
\]
\end{thm}

\begin{thm}\label{thm:LP2}
Under the same conditions in Theorem~\ref{thm:LP1}, we have
\[
\sum_{\substack{D_i\in \Dyck_{\lambda}(s_i,t_i) \\ D_1<\dots<D_k}} q^{\sum_{i=1}^k |\VV(D_i)|} = q^{-\sum_{i=1}^k r_i} \det \left(
F_\lambda(s_i,t_j;q)\right)_{i,j=1}^k,
\]
where
\[
F_\lambda(s_i,t_j;q) = \sum_{D\in\Dyck_\lambda(s_i,t_j)} q^{|\VV(D)|}.
\]
\end{thm}

Since
\[
p(\lm) = 2^{|\lm|}\sum_{\substack{D_i\in \Dyck_{\lambda}(s_i,t_i) \\ D_1<\dots<D_k}} 2^{-\sum_{i=1}^k |\VV(D_i)|},
\]
by substituting $q=1/2$ in Theorem~\ref{thm:LP2}, we obtain 
\[
p(\lm) = 2^{|\lm|+\sum_{i=1}^k r_i} \det \left(
F_\lambda(s_i,t_j;1/2)\right)_{i,j=1}^k,
\]  
which can be rewritten as follows. 

\begin{cor}\label{cor:LP2}
Under the same conditions in Theorem~\ref{thm:LP1}, we have
\[
p(\lm) = 2^{\sum_{i=1}^k r_i} \det \left(
p(L_\lambda(s_i,t_j))\right)_{i,j=1}^k.
\]  
\end{cor}

If $\lm=\delta_{n+2k}/\delta_n$ in Theorem~\ref{thm:LP1}, we obtain
Theorem~\ref{conj:9.6}. If $\lm=\delta_{n+2k+1}/\delta_n$ in Theorem~\ref{thm:LP1}, 
using the Kreiman outer decomposition as shown in Figure~\ref{fig:E_odd}, we obtain
the following corollary.

\begin{figure}[h]
  \centering
\includegraphics{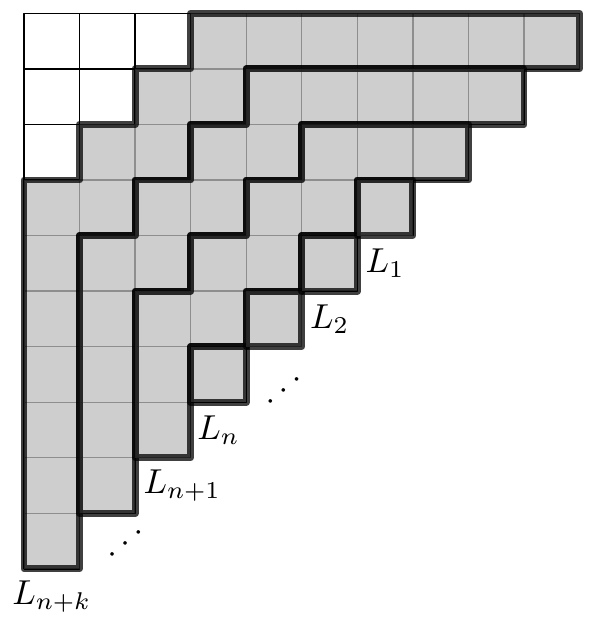} 
  \caption{The Kreiman outer decomposition $L_1,\dots,L_{n+k}$ of $\delta_{n+2k+1}/\delta_n$. The label $L_i$ is written below the starting cell of it.}
  \label{fig:E_odd}
\end{figure}

\begin{cor}\label{cor:E_odd}
We have
\[
\sum_{\pi\in \RPP(\delta_{n+2k+1}/\delta_n)} q^{|\pi|} 
=q^{-\frac{k(k+1)(6n+8k+1)}{6}} \det\left( \frac{E^*_{2i+2\overline{j}+1}(q)}{(q;q)_{2i+2\overline{j}+1}} \right)_{i,j=1}^{n+k},
\]  
where $E^*_t(q)=0$ if $t<0$ and 
\[
\overline j =
\begin{cases}
j-n-1 & \mbox{if $j>n$},\\  
-j & \mbox{if $j\le n$}.
\end{cases}
\]
\end{cor}

Similarly, as a special case of  Corollary~\ref{cor:LP2}, we obtain a determinant formula for the number of pleasant diagrams of shape
$\delta_{n+2k+1}/\delta_n$.
\begin{cor}\label{cor:E_odd2}
We have
\[
p(\delta_{n+2k+1}/\delta_n)
=2^{\binom{k+1}2} \det\left( \mathfrak{s}_{i+\overline{j}} \right)_{i,j=1}^{n+k},
\]  
where 
\[
\mathfrak{s}_m = 
\begin{cases}
  0 & \mbox{if $m<0$,}\\
  2 & \mbox{if $m=0$,}\\
  2^{m+2}s_m & \mbox{if $m>0$,}\\
\end{cases}
\]
and $s_m$ is the little Schr\"oder number.
\end{cor}

\begin{exam}
If $n=2$ and $k=1$, we have
\[
p(\delta_{n+2k+1}/\delta_{n}) = p(\delta_{5}/\delta_2) = 2^8+2^9=3\cdot 2^8,
\]
and
\[
2^{\binom{k+1}2} \det\left( \mathfrak{s}_{i+\overline{j}} \right)_{i,j=1}^{n+k}
=2 \det\left(
  \begin{matrix}
    \mathfrak{s}_0 & \mathfrak{s}_{-1} & \mathfrak{s}_1 \\
    \mathfrak{s}_1 & \mathfrak{s}_{0} & \mathfrak{s}_2 \\
    \mathfrak{s}_2 & \mathfrak{s}_{1} & \mathfrak{s}_3\\
  \end{matrix}
 \right)
=2 \det\left(
  \begin{matrix}
    2&0&8\\
    8&2&48\\
    48&8&352\\
  \end{matrix}
 \right) = 3\cdot 2^8.
\]
\end{exam}

Note that if $\alpha$ is a reverse hook, i.e., $\alpha=(b^a)/((b-1)^{a-1})$, it is easy to see that
\[
\sum_{\pi\in \RPP(\alpha)}  q^{|\pi|} = \sum_{t\ge0} q^t \qbinom{a+t-1}{t}\qbinom{b+t-1}{t},
\]
where 
\[
\qbinom{n}{m} = \frac{(q;q)_n}{(q;q)_m(q;q)_{n-m}}.
\]
Therefore, if $\lm$ is a thick reverse hook $((b+k)^{a+k})/(b^a)$, we obtain the following formula as a corollary of Theorem~\ref{thm:LP1}. 

\begin{cor}\label{cor:rhook}
Let $\lambda=((b+k)^{a+k})$ and $\mu=(b^a)$. Then
\[
\sum_{\pi\in \RPP(\lm)}  q^{|\pi|} = q^{-\frac{k(k-1)(3a+3b+4k+1)}6}
\det\left( \sum_{t\ge0} q^t \qbinom{a+t+i-1}{t}\qbinom{b+t+j-1}{t}
\right)_{i,j=1}^k.
\]  
\end{cor}

\begin{remark}
Theorem~\ref{thm:LP1} is not applicable to $\lm$ for arbitrary rectangular shapes $\lambda$ and $\mu$. 
For example, if $\lm=(6,6,6,6)/(3,3)$, the Kreiman outer decomposition of $\lm$
has three $\lambda$-Dyck paths $L_1<L_2<L_3$ as shown in Figure~\ref{fig:rect}. 
Since $L_2$ does not start with an up step, the conditions in Theorem~\ref{thm:LP1} are not satisfied.
\end{remark}

\begin{figure}[h]
  \centering
\includegraphics{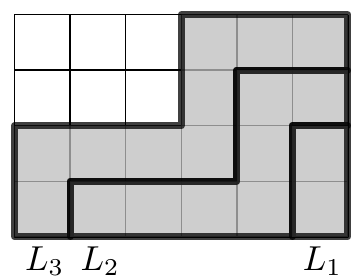} 
  \caption{The Kreiman outer decomposition of $(6,6,6,6)/(3,3)$.}
  \label{fig:rect}
\end{figure}

Let $\lambda=(b^a)$,  $s$ the southmost cell in the first column of $\lambda$ and $t$ the eastmost cell in the first row of $\lambda$. 
Then every $\lambda$-Dyck path in $\Dyck_\lambda(s,t)$ is completely determined by its valleys. By considering the valleys
of the $\lambda$-Dyck paths, one can prove that, for a reverse hook $\alpha=(b^a)/((b-1)^{a-1})$, 
\[
p(\alpha)=\sum_{t\ge0} 2^{a+b-t-1}\binom{a-1}t \binom{b-1}t.
\]
Thus, as a special case of Corollary~\ref{cor:LP2}, we have a determinant formula for $p(\lm)$ when  $\lm$ is a thick reverse hook $((b+k)^{a+k})/(b^a)$. 
\begin{cor}\label{cor:rhook2}
Let $\lambda=((b+k)^{a+k})$ and $\mu=(b^a)$. Then
\[
p(\lm) = 2^{\binom k2} \det\left( \sum_{t\ge0} 2^{a+b+i+j-t-1}\binom{a+i-1}t \binom{b+j-1}t
\right)_{i,j=1}^k.
\]  
\end{cor}

\section*{Acknowledgement}

The authors would like to thank Alejandro Morales for his helpful comments which motivated Section~\ref{sec:lascoux-pragacz-type}.


\begin{thebibliography}{10}

\bibitem{Bhatnagar}
G.~Bhatnagar.
\newblock How to prove {R}amanujan's $q$-continued fractions.
\newblock Ramanujan 125, 49--68, {\it Contemp. Math.}, 627, Amer. Math. Soc.,
  Providence, RI, 2014.

\bibitem{Euler1788}
L.~Euler.
\newblock De transformatione seriei divergentis $1-mx + m(m + n)x^2-m(m + n)(m
  + 2n)x^3 +m(m+n)(m+2n)(m+3n)x^4+ etc.$ in fractionem continuam.
\newblock {\em Nova Acta Academiae Scientarum Imperialis Petropolitinae},
  2:36--45, 1788.

\bibitem{Flajolet1980}
P.~Flajolet.
\newblock Combinatorial aspects of continued fractions.
\newblock {\em Discrete Math.}, 32(2):125--161, 1980.

\bibitem{Foata68}
D.~Foata.
\newblock On the {N}etto inversion number of a sequence.
\newblock {\em Proc. Amer. Math. Soc.}, 19:236--240, 1968.

\bibitem{Frame1954}
J.~S. Frame, G.~d.~B. Robinson, and R.~M. Thrall.
\newblock The hook graphs of the symmetric groups.
\newblock {\em Canadian J. Math.}, 6:316--324, 1954.

\bibitem{Fulmek2000}
M.~Fulmek.
\newblock A continued fraction expansion for a {$q$}-tangent function.
\newblock {\em S{\'e}m. Lothar. Combin.}, 45:Art.\ B45b, 5, 2000/01.

\bibitem{GesselViennot}
I.~Gessel and G.~Viennot.
\newblock Binomial determinants, paths, and hook length formulae.
\newblock {\em Adv. in Math.}, 58(3):300--321, 1985.

\bibitem{Huber2010}
T.~Huber and A.~J. Yee.
\newblock Combinatorics of generalized {$q$}-{E}uler numbers.
\newblock {\em J. Combin. Theory Ser. A}, 117(4):361--388, 2010.

\bibitem{IkedaNaruse09}
T.~Ikeda and H.~Naruse.
\newblock Excited {Y}oung diagrams and equivariant {S}chubert calculus.
\newblock {\em Trans. Amer. Math. Soc.}, 361(10):5193--5221, 2009.

\bibitem{Jackson1904}
F.~Jackson.
\newblock A basic sine cosine with symbolic solutions of certain differential
  equations.
\newblock {\em Proc. Edinb. Math. Soc.}, 22:28--39, 1904.

\bibitem{Kreiman}
V.~Kreiman.
\newblock Schubert classes in the equivariant {K}-theory and equivariant
  cohomology of the {G}rassmannian.
\newblock {\bf arXiv:math.AG/0512204}.

\bibitem{Lascoux1988}
A.~Lascoux and P.~Pragacz.
\newblock Ribbon {S}chur functions.
\newblock {\em European J. Combin.}, 9(6):561--574, 1988.

\bibitem{Lindstrom}
B.~Lindstr\"om.
\newblock On the vector representations of induced matroids.
\newblock {\em Bull. London Math. Soc.}, 5:85--90, 1973.

\bibitem{MPP1}
A.~Morales, I.~Pak, and G.~Panova.
\newblock Hook formulas for skew shapes {I}. $q$-analogues and bijections.
\newblock \url{https://arxiv.org/abs/1512.08348}.

\bibitem{MPP2}
A.~Morales, I.~Pak, and G.~Panova.
\newblock Hook formulas for skew shapes {II}. {C}ombinatorial proofs and
  enumerative applications.
\newblock \url{https://arxiv.org/abs/1610.04744}.

\bibitem{Naruse}
H.~Naruse.
\newblock Schubert calculus and hook formula.
\newblock Talk slides at 73rd S\'em. Lothar. Combin., Strobl, Austria, 2014;
  available at
  \url{https://www.emis.de/journals/SLC/wpapers/s73vortrag/naruse.pdf}.

\bibitem{Prodinger2000}
H.~Prodinger.
\newblock Combinatorics of geometrically distributed random variables: new
  {$q$}-tangent and {$q$}-secant numbers.
\newblock {\em Int. J. Math. Math. Sci.}, 24(12):825--838, 2000.

\bibitem{Prodinger2008}
H.~Prodinger.
\newblock A continued fraction expansion for a {$q$}-tangent function: an
  elementary proof.
\newblock {\em S{\'e}m. Lothar. Combin.}, 60:Art. B60b, 3, 2008/09.

\bibitem{Stanley72}
R.~P. Stanley.
\newblock {\em Ordered structures and partitions}.
\newblock American Mathematical Society, Providence, R.I., 1972.
\newblock Memoirs of the American Mathematical Society, No. 119.

\bibitem{Stanley76}
R.~P. Stanley.
\newblock Binomial posets, {M}{\"o}bius inversion, and permutation enumeration.
\newblock {\em J. Combinatorial Theory Ser. A}, 20(3):336--356, 1976.

\bibitem{Stanley2010}
R.~P. Stanley.
\newblock A survey of alternating permutations.
\newblock In {\em Combinatorics and graphs}, volume 531 of {\em Contemp.
  Math.}, pages 165--196. Amer. Math. Soc., Providence, RI, 2010.

\bibitem{EC1}
R.~P. Stanley.
\newblock {\em Enumerative Combinatorics. {V}ol. 1, second ed.}
\newblock Cambridge University Press, New York/Cambridge, 2011.

\end{thebibliography}

\end{document}